\documentclass{amsart}
\usepackage[utf8]{inputenc}

\usepackage[T1]{fontenc}
\usepackage{amsmath,amssymb,amsthm}
\usepackage{graphicx}
\usepackage[dvipsnames]{xcolor}
\usepackage{verbatim}
\usepackage{lmodern}
\usepackage{tikz-cd}
\usepackage{hyperref}
\usepackage[english]{babel}
\usepackage{color}
\usepackage[margin=1in]{geometry}
\usepackage{cite}
\usepackage{mathtools}

\newtheorem{theorem}{Theorem}[section]
\newtheorem*{theorem*}{Theorem}

\newtheorem{lemma}[theorem]{Lemma}
\newtheorem{corollary}[theorem]{Corollary}
\newtheorem{remark}[theorem]{Remark}

\newtheorem{proposition}[theorem]{Proposition}

\newcommand{\Hom}{\mathsf{Hom}}
\newcommand{\RR}{\mathbb{R}}
\newcommand{\CC}{\mathbb{C}}

\newcommand{\ZZ}{\mathbb{Z}}

\newcommand{\GL}{\mathsf{GL}}

\newcommand{\SO}{\mathsf{SO}}

\renewcommand{\O}{\mathsf{O}}

\newcommand{\E}{\mathcal{E}}

\newcommand{\CP}{\mathbb{CP}}

\newcommand{\F}{\mathcal{F}}

\newcommand{\abs}[1]{\left|#1\right|}

\newcommand{\minus}{\setminus}

\newcommand{\St}{\mathsf{St}}

\renewcommand{\E}{\mathcal{E}}
\newcommand{\Gr}{\mathsf{Gr}}
\newcommand{\Grtil}{\Gr}
\newcommand{\rank}{\operatorname{Rank}}
\newcommand{\tr}{\operatorname{Tr}}

\newcommand{\Id}{\operatorname{Id}}
\newcommand{\wbar}{\overline{w}}

\usepackage[normalem]{ulem}
\newcounter{RCnumber}

\title[The Equal-Norm Grassmannian]{The Equal-Norm Grassmannian: Topology and Applications to Frame Fields}
\author{Samuel A.\ Ballas}
\address{Department of Mathematics, Florida State University, Tallahassee, FL 32306}
\email{sballas@fsu.edu}

\author{Ferhat Karabatman}
\address{Department of Mathematics, Florida State University, Tallahassee, FL 32306}
\email{fk22@fsu.edu}

\author{Tom Needham}
\address{Department of Mathematics, Florida State University, Tallahassee, FL 32306}
\email{tneedham@fsu.edu}

\author{Clayton Shonkwiler}
\address{Department of Mathematics, Colorado State University, Fort Collins, CO 80523}
\email{clayton.shonkwiler@colostate.edu}
\date{\today}

\begin{document}

\begin{abstract}
    Given an orientable rank-$d$ vector bundle over a $d$-manifold, the Euler class is the unique obstruction to finding a nowhere-zero section. An equal norm Parseval frame is a certain collection of sections, one of whose properties is that the lengths of each vector at each point are equal. This makes the Euler class an obvious obstruction to finding such a frame. In this work we show that this is the only obstruction so long as the number of sections is at least $2d+2$. We also identify precise obstructions when the number of sections is $2d$  and produce examples showing that for smaller numbers of sections it is sometimes impossible to find Parseval frame fields.
\end{abstract}

\maketitle

\section{Introduction}\label{sec:introduction}

Given a rank-$d$ vector bundle $E$ over a manifold $M$, a classical problem in algebraic topology involves investigating collections $\{\sigma_1,\ldots,\sigma_n\}$ of sections with various pointwise linear algebraic properties. Specific examples include nowhere zero sections and collections of sections which form linearly independent sets of vectors at each point. In this work we examine two additional pointwise geometric conditions, under the assumption that the vector bundle is equipped with a fiberwise inner product. The first is that of being \emph{equal norm}. This simply means that every vector of every section has the same fixed length at every point. A slightly more exotic condition is that the collection of sections be \emph{Parseval}, which says that for each $p \in M$, the associated \emph{analysis operator}
\[
E_p \to \RR^n:v \mapsto (\langle v, \sigma_i(p) \rangle)_{i=1}^n
\]
on the fiber $E_p$ has all singular values equal to one; for $n \geq d$, this is a generalization of the condition that the sections form an orthonormal basis at each point. Our interest in this condition is motivated as follows. When $M$ is a single point, $E$ is just a vector space, and a collection of sections is the same as a collection of vectors in $E$. In this setting, such a spanning collection is referred to as a \emph{frame}. Equal norm Parseval frames are commonly studied objects in the field of frame theory,  and have important applications to signal processing \cite{holmes2004optimal, casazza2003equal, bodmann2005frames, vershynin2005frame,goyal2001quantized,benedetto2006second,bodmann2007frame}. The vector bundle setting considered here is therefore a generalization of a rather natural condition arising in harmonic analysis. Borrowing this terminology, we refer to a collection of sections of $E$ which is pointwise spanning as a frame for $E$ and to the evaluation of such a collection at a point $p$ as a frame for $E_p$. There has been significant interest in questions of existence of frames with prescribed properties for vector bundles~\cite{frank2002frames,packer2004projective,freeman2012moving,freeman2012moving2,BallasNeedhamShonkwiler2025,favazza2026phase}, partially motivated by problems in mathematical physics~\cite{kuchment2008tight,auckly2018parseval,kuchment2016overview,cornean2019parseval,monaco2023topology}.

A simple exercise in linear algebra shows that a collection of sections of size $d$ is Parseval if and only if it is a pointwise orthonormal basis. Such a collection can exist if and only if the bundle $E$ is trivial, so one immediately sees that the topology of the bundle can place restrictions on the types of frame fields that the bundle can admit. Understanding these topological obstructions is the focus of this paper. The main results described below restrict to the setting where the dimension of the manifold $M$ is equal to the rank $d$ of the vector bundle $E$. To avoid the need to deal with special cases at certain points, we assume throughout the paper that $d > 2$. Furthermore, unless explicitly stated otherwise, we assume that $M$ is a closed, connected, orientable smooth manifold, and sections are assumed to be continuous.

One example of such a topological obstruction is the Euler class of an orientable vector bundle. This is classically known to be the primary obstruction to the existence of a nowhere zero section. If $\{\sigma_1,\ldots, \sigma_n\}$ is a collection of non-trivial equal norm sections, each $\sigma_i$ is a nowhere zero section. Thus the Euler class is an obvious obstruction to the existence of equal norm Parseval frame fields. Our first main result shows that when $n$ is sufficiently large (compared to $d$), this is the only obstruction.

\begin{theorem}\label{thm:main_thm}
    Let $E\to M$ be an orientable rank-$d$ vector bundle over a $d$-dimensional manifold and let $n\geq 2d+2$. $E$ admits an equal norm Parseval frame field of size $n$ if and only if the Euler class of $E$ is trivial. 
\end{theorem}

For smaller $n$, there are additional obstructions in the form of Stiefel–Whitney classes of certain complementary bundles. Specifically, a Parseval frame field of size $n$ defines an isometric bundle embedding of $E$ into $M\times \RR^n$. Such a bundle admits a complement $E^\perp$, which is a rank-$(n-d)$ bundle. The isomorphism class of this bundle depends on the choice of frame, but different choices of frames give rise to \emph{stably equivalent} bundles. More precisely, for each choice we have $E\oplus E^\perp\cong \epsilon^n$, where $\epsilon^n$ is the trivial rank-$n$ bundle over $M$, and so any such choice of complement is a representative of the inverse of the reduced real $K$-theory class of $E$.  Thus stable characteristic classes, such as Stiefel–Whitney classes, give honest invariants of $E$ that do not depend on the particular choice of Parseval frame. More precisely, since $E^\perp$ is a stable inverse of $E$ it follows that $w(E^\perp)=w(E)^{-1}$ (i.e., the inverse in the cohomology ring). Henceforth, we will use the notation $\wbar(E)=w(E)^{-1}$, and $\wbar_k(E)$ for the corresponding graded piece of the total Stiefel–Whitney class. 

We also study an additional critical dimension $n=2d$. In this case the bundle $E^\perp$ is rank-$d$. Since the base manifold has dimension $d$, it follows that there is a top Stiefel–Whitney class $\wbar_d(E)\in H^d(M,\ZZ_2)\cong \ZZ_2$. In this critical dimension we see that this Stiefel–Whitney class is the only additional obstruction.

\begin{theorem}\label{thm:critical_dimension_theorem}
    Let $d>2$ and let $E\to M$ be an orientable rank-$d$ vector bundle over a closed orientable $d$-manifold. Then $E$ admits an equal norm Parseval frame of size $2d$ if and only if $e(E)=0$ and $\wbar_d(E)=0$. 
\end{theorem}

When $d$ is odd, the antipodal map is orientation reversing, and this implies that the Euler class is an element of 2-torsion. If we further assume that the base is orientable then $H^d(M,\ZZ)$ is torsion free, and hence the Euler class is forced to vanish.

From the work of \cite{BallasNeedhamShonkwiler2025} (see Theorem \ref{thm:parseval_frames_exist}) $E$ admits a Parseval frame field of size $2d$, and the corresponding complement $E^\perp$ has rank $d$. By the argument above we have $e(E^\perp)=0$, and since $\wbar_d(E)=w_d(E^\perp)$ is the mod 2 reduction of $e(E^\perp)$, it follows that $\wbar_d(E)=0$, and we get the following corollary.

\begin{corollary}\label{cor:odd_case}
    Let $d$ be odd and $E\to M$ a rank-$d$ orientable vector bundle over a $d$-dimensional smooth orientable manifold. $E$ always admits an equal norm Parseval frame of size $2d$.
\end{corollary}

Our methods also provide partial results in the intermediate case \(n=2d+1\). In particular, we establish that the vanishing of \(e(E)\) and \(\bar w_d(E)\) is necessary for the existence of an equal-norm Parseval frame field.

In \cite{BallasNeedhamShonkwiler2025} the authors show that for any orientable rank-$d$ vector bundle over a $k$-manifold, it is possible to find a Parseval frame field of size $n=d+k$. In Section~\ref{sec:examples} we will provide examples of vector bundles that do not admit Parseval frame fields of size $d+k-1$. These results fit nicely together in the following picture. There is a stable range where $n\geq 2d+2$ where a rank-$d$ vector bundle over an orientable $d$-manifold admits an equal norm Parseval frame field of size $n$ precisely when $e(E)=0$.  On the other hand there is a subcritical range $n<2d$, where there is not even a guarantee that one can find a Parseval frame field. 

\subsection{Overview of techniques}

It is well known that orientable vector bundles are classified by the infinite oriented Grassmannian $\Grtil_d$. More precisely, there is a universal tautological rank-$d$ vector bundle $\tau_d$ over $\Grtil_d$. Every oriented vector bundle $E\to M$ arises from pulling back $\tau_d$ by some map $f:M\to \Grtil_d$. Furthermore, homotopic maps give rise to isomorphic vector bundles. When one has a Parseval frame field of size $n$, it gives rise to a map $f:M\to \Grtil_{d,n}$ into a finite Grassmannian. This map can be regarded as a classifying map (see Lemma \ref{lem:classifying_map}) if we regard $\Grtil_{d,n}\hookrightarrow \Grtil_d$. Said a bit differently, if one can find a Parseval frame field of size $n$, then it is possible to homotope a classifying map to have image in $\Grtil_{d,n}$. Furthermore, there is a subspace $\E_{d,n}\subset \Grtil_{d,n}$ parametrizing equivalence classes of equal-norm Parseval frames, and the problem of finding an equal norm Parseval frame of size $n$ can then be rephrased as deforming a classifying map of $E$ to have image in $\E_{d,n}$ (see Proposition \ref{prop:homotope_to_equal_norm}). 

Deforming maps to have image contained in a subspace is a classical algebraic topological problem. There is a well established obstruction theory that provides families of cohomological invariants that obstruct such deformations. More precisely, these cohomological invariants live in the groups $H^k(M;\pi_k(\Grtil_{d,n},\E_{d,n}))$. As we will see, once one gets to the critical range $n\geq 2d$, the pair $(\Grtil_{d,n},\E_{d,n})$ is $(d-1)$-connected.  We can also calculate geometrically motivated generators for the first non-trivial group $\pi_d(\Grtil_{d,n},\E_{d,n})$ (see Theorem \ref{thm:d-1_connected_plus_generators}). When $n\neq 2d+1$ we show that these generators freely generate $\pi_d(\Grtil_{d,n},\E_{d,n})$. However, when $n=2d+1$ we show that relations occur among the generators.

When $n=2d+1$ we are not able to completely calculate the relevant homotopy group. This is the reason that we can't currently extend all of our results to the $n=2d+1$ setting.

Directly analyzing the topology of the equal norm locus is difficult. Instead, we introduce an explicit unstable subset \(Z\) of the Grassmannian and show that the equal-norm locus is a strong deformation retract of \(\Grtil_{d,n}\setminus Z\).  The components of \(Z\) are Schubert varieties whose codimensions can be computed explicitly, allowing the low-dimensional relative topology of \((\Grtil_{d,n},\mathcal E_{d,n})\) to be determined by general-position arguments. Roughly speaking, the generators of $\pi_d(\Grtil_{d,n},\E_{d,n})$ come from ``meridians'' around the top dimensional pieces of $Z$. This meridional description of the generators allows us to identify the coordinates of the primary obstruction with classical characteristic classes.  In the stable range $n\geq 2d+2$ the generators can be realized as meridional disks around the zero locus of a section of the tautological bundle over $\Gr_{d,n}$, and hence correspond to the Euler class. In the critical dimension $n=2d$, additional codimension-$d$ strata appear. These can again be represented by meridional disks around the zero locus of a section of the orthogonal complement of the tautological bundle of $\Grtil_{d,n}$. The resulting obstruction theory is thus related to the Euler class of a complementary bundle, and ultimately to its top Stiefel–Whitney class.

\subsection{Relations to Previous Work}\label{sec:previous_work}

The present work is also closely related to \cite{BallasNeedhamShonkwiler2025}, where the authors study the existence of Parseval frame fields on vector bundles without imposing the equal-norm condition. In particular, the authors show that an orientable rank-$d$ bundle over a $k$-manifold admits a Parseval frame field of size $d+k$, and they investigate the sharpness of this bound. The present paper studies the additional topological obstructions arising from the equal-norm condition. The complementary-bundle viewpoint developed in \cite{BallasNeedhamShonkwiler2025} also plays an important role in our analysis of the critical case $n=2d$.

To the best of our knowledge, the only previous work attempting to construct equal norm Parseval frames on vector bundles is~\cite{freeman2012moving2}. The authors of that paper show that the tangent bundle of every odd-dimensional sphere admits an equal norm Parseval frame of some size.\footnote{Technically, they study unit norm tight frames, but up to uniform rescaling these are the same as equal norm Parseval frames.}  Their construction is fairly explicit and they are able to produce equal norm Parseval frames of size $(2n-1)2^{n-1}$ on the $(2n-1)$-sphere. Our results differ from theirs in several important ways. First, our result holds for arbitrary orientable rank-$d$ vector bundles over arbitrary $d$-manifolds. Second, the sizes of the frame fields constructed in \cite{freeman2012moving2} grow exponentially with the dimension of the sphere, whereas the frames in our results grow linearly in dimension. Since the 1- and 3-dimensional spheres are well known to be parallelizable, the first interesting case is the 5-sphere. The construction in \cite{freeman2012moving2} gives an equal norm Parseval frame field of size 20, whereas Corollary~\ref{cor:odd_case} guarantees the existence of an equal norm Parseval frame field of size 10. 

Our analysis of the equal-norm locus is also motivated by the work of Caine, Needham, and Shonkwiler \cite{caineNormalizingParsevalFrames2024,CaineNeedhamShonkwiler2025}, who use an energy functional and its gradient flow to study equal-norm frames. Their description of the unstable locus and the associated deformation by gradient flow inspires the approach of Section~\ref{sec:GIT}. We recast this picture intrinsically on the Grassmannian, where the unstable locus can be described using geometric invariant theory and its components identified as Schubert varieties.

The complementary-bundle questions arising in the critical range are also related to classical problems concerning vector bundles within a fixed stable equivalence class. Malyi \cite{Malyi1974}, for example, studied the possible Euler classes of fixed-rank bundles stably equivalent to a given bundle. His results, particularly in the metastable range, are closely related to the complementary-bundle phenomena that appear here.

\subsection{Outline}

We begin with Section \ref{sec:background} where we discuss background material related to finite frames in Hilbert spaces and obstruction theory. Section \ref{sec:vector_bundles} provides additional background related to orientable vector bundles. A main focus of this section is on the construction of vector bundles using classifying maps. It also introduces the philosophy that finding Parseval and equal norm Parseval frame fields is equivalent to deforming classifying maps to have image in certain subsets of the Grassmannian $\Grtil_d$ (see Lemma \ref{lem:classifying_map} and Proposition \ref{prop:homotope_to_equal_norm}). This section also includes background material related to the Euler class. 

Sections \ref{sec:GIT} and \ref{sec:topology_of_pair} are the technical heart of the paper. In these sections we study the topology of the pair $(\Grtil_{d,n},\E_{d,n})$. The first main result of this part is Theorem \ref{thm:homotop_equivalence}, which shows that there is a strong deformation retraction from $\Grtil_{d,n}\minus Z$ to $\E_{d,n}$, where $Z$ is a closed subset of $\Grtil_{d,n}$ obtained as a union of explicit Schubert varieties. The Schubert varieties comprising $Z$ end up being the unstable locus of a certain algebraic group action and the strong deformation retraction comes from applying Kempf--Ness theory to the gradient flow along a certain ``energy function.''

By identifying the codimension of the various pieces of $Z$, we are able to prove Theorem \ref{thm:d-1_connected_plus_generators} that establishes the $(d-1)$-connectivity of this pair. Subsequently, Lemma \ref{lem:enumerate_top_dim_pieces} identifies a natural generating set for $\pi_d(\Grtil_{d,n},\E_{d,n})$, and Proposition \ref{prop:homotopy_of_stiefel_pair} shows that when $n\neq 2d+1$ these generators freely generate $\pi_d(\Grtil_{d,n},\E_{d,n})$ as an abelian group. 

 Section \ref{sec:proof_of_main} uses the previous description of homotopy groups to explicitly describe the relevant obstructions for deforming classifying maps. Once these have been identified the proof of  Theorem \ref{thm:main_thm} quickly follows.
 
 Section \ref{sec:complementary_bundles} introduces complementary bundles and the importance of their Euler and Stiefel–Whitney classes. This section deals with two related problems: first, when one can find a Parseval frame field of size $d+k-1$ on a rank-$d$ bundle $E\to M$ over a $k$-dimensional manifold; and second, if $k=d$, when a Parseval frame of size $n=2d$ can be deformed (through Parseval frames) to be equal norm. We also derive a necessary Stiefel--Whitney condition for the existence of an equal norm Parseval frame of size $2d+1$. For the two deformation problems, the additional obstruction comes from the Euler class of a complementary bundle (see Propositions \ref{prop:euler_obstruction_for_parseval_frames}
 and \ref{prop:perp_has_trivial_euler_class}). We then prove complementary results (see Theorems \ref{thm:d+k-1_obstruction} and \ref{thm:2d_obstruction}) showing that one can find the desired Parseval frames when the corresponding Stiefel–Whitney class vanishes. This is done via surgery techniques that allow one to modify a complementary bundle and trivialize the corresponding Euler class. Putting this together proves Theorem \ref{thm:critical_dimension_theorem}. 

Finally, Section \ref{sec:examples} contains families of examples that show that the dimension restrictions in Theorem \ref{thm:main_thm} and Theorem 4.1 of \cite{BallasNeedhamShonkwiler2025} are sharp.

\subsection*{Declaration of AI usage}
Generative AI tools were used to assist with manuscript preparation, including editorial feedback, proofreading, and discussion of mathematical arguments. All mathematical results and proofs were developed, evaluated, and independently verified by the authors, who retain full intellectual responsibility for the work.

\section{Background}\label{sec:background}
\subsection{Finite frames}
In this section we provide some background on finite frames in Hilbert spaces. All of this material is standard and can be found in, e.g.,  \cite{CasazzaKutyniok2013,waldron2018introduction}. Let $V$ be a finite-dimensional real Hilbert space. A collection $\F=\{v_1,\ldots,v_n\}$ is called a \emph{frame} if it spans $V$. A frame $\mathcal{F}$ gives rise to a \emph{analysis operator} $A_\F:V\to \RR^n$ given by the formula
\[
A_\F(v)=(\langle v,v_i\rangle)_{i=1}^n.
\]
In words, the analysis operator pairs $v$ with each of the vectors in $\F$ and records the results in an $n$-tuple. The adjoint of the analysis operator $A_{\F}^\ast:\RR^n\to V$ is called the \emph{synthesis operator} and is given by the formula
\[
(c_i)_{i=1}^n\mapsto \sum_{i=1}^nc_iv_i.
\]
These operators compose to give the \emph{frame operator} $A_\F^\ast A_\F:V\to V$. A special class of frames are the  \emph{tight frames}, when $A_\F^\ast A_\F=\lambda I$ for some $\lambda > 0$. Tight frames with $\lambda = 1$ are called \emph{Parseval frames}. A fundamental result in frame theory due to Han and Larson \cite{han2000frames} shows that orthogonally projecting an orthonormal basis onto a subspace gives a Parseval frame, and conversely, every Parseval frame arises from this construction.

In the context of signal processing, we can think of the frame vectors as \emph{sensors} and the analysis operator as collecting a series of measurements. In this setting Parseval frames have a simple strategy for recovering a vector from its collection of measurements, namely $v=\sum_{i=1}^n\langle v,v_i\rangle v_i$. As real-world measurements are subject to noise, it is useful for one to record a  highly redundant set of measurements, i.e., for $n \gg \mathrm{dim}(V)$, which can mitigate recovery errors. It is well known that Parseval frames provide optimally accurate signal recovery under additive measurement noise models~\cite{goyal2001quantized}.

Another important class of frames are the \emph{equal norm} frames defined by the property that $\|v_i\|=\|v_j\|$ for each pair $1\leq i,j\leq n$.  We will primarily be interested in frames that are both Parseval and equal norm. The entries of the \emph{Gram matrix} $A_\F A_\F^\ast$ record the pairwise inner products of the frame vectors. In particular the diagonal entries record the squared norms of the frame vectors, so equal norm frames will have Gram matrices with constant diagonal. Since $A_\F A_\F^\ast$ has the same trace as the frame operator $A_\F^\ast A_\F$, we see that the equal norm Parseval condition forces $\|v_i\|^2=\frac{\dim(V)}{n}$, so the equal norm value is completely determined. 

Another way of thinking of the frame operator is as 
\[
A_\F^\ast A_\F =\sum_{i=1}^n v_iv_i^\ast.
\]

As such, the frame operator is a weighted sum of orthogonal projection operators onto the $v_i$ weighted by the norm squared of $v_i$. With this perspective a Parseval frame corresponds to a decomposition of the identity into a weighted sum of projection operators and an equal norm Parseval frame represents the most democratic possible choice of weights. 

Given $v,w\in V$ and a Parseval frame $\F$ we have 
\[
\langle v,w\rangle=\langle A_\F^\ast A_\F v,w\rangle=\langle A_\F v,A_\F w\rangle. 
\]
Hence the Parseval condition is equivalent to the analysis operator being an isometric embedding. If we fix an orthonormal basis $\{e_1,\ldots, e_d\}$ of $V$ then $\{A_\F e_1,\ldots, A_\F e_d\}$ is an ordered orthonormal set in $\RR^n$. Thus, we have an identification of the space of Parseval frames of size $n$ in a $d$-dimensional Hilbert space with the \emph{Stiefel manifold} $\St_{d,n}$ consisting of ordered orthonormal sets in $\RR^n$ of size $d$. Under this identification, we denote by $\E^S_{d,n} \subset \St_{d,n}$ the set corresponding to equal norm Parseval frames (with the superscript $S$ reminding us that this is a subset of the Stiefel manifold).

Closely related to the Stiefel manifold is the \emph{Grassmannian} $\Grtil^u_{d,n}$ of $d$-dimensional subspaces in $\RR^n$; we use the superscript $u$ to emphasize that this is the \emph{unoriented Grassmannian}, as opposed to the \emph{oriented Grassmannian} $\Grtil_{d,n}$ that will be introduced below. There is a natural map $\St_{d,n}\to \Grtil^u_{d,n}$ obtained by mapping a collection of $d$ orthonormal vectors to their span. This map is well known to be a principal $\O(d)$-bundle. The action of $B\in \O(d)$ is given by precomposing the analysis operator with $B$. Since both $B$ and the analysis operator are isometric, it follows that $\E^S_{d,n}$ is $\O(d)$-invariant. Since the action is free we get another principal $\O(d)$ bundle $\E^S_{d,n}\to \E^u_{d,n}$, where $\E^u_{d,n}\subset \Grtil^u_{d,n}$ is the image under projection. Furthermore, the Parseval frames $\F$ and $\F\cdot B$ have the same Gram matrix.  It is a standard linear algebraic fact that the Gram matrix is the coordinate representation of the orthogonal projection operator from $\RR^n$ onto $Im(A_\F)$. Thus, if $\F$ and $\F'$ have the same Gram matrix then their analysis operators must have the same image and hence must differ by precomposition with some $B\in \O(d)$. As a consequence, we can regard elements of $\Grtil^u_{d,n}$ as equivalence classes of Parseval frames with the same Gram matrix. One useful consequence of this observation is that it provides a more concrete description of $\E^u_{d,n}$. If $F\in \Grtil^u_{d,n}$ is a $d$-plane then $F\in \E^u_{d,n}$ if and only if the diagonal entries of the orthogonal projection matrix $P_F$ are all equal to $\frac{d}{n}$.

More relevant for our purposes is the closely related \emph{oriented Grassmannian} $\Grtil_{d,n}$ consisting of $d$-planes in $\RR^n$ along with a choice of orientation. An ordered orthonormal set defines an orientation on its span and so we also get a map $\St_{d,n}\to \Grtil_{d,n}$. This map is an $\SO(d)$ principal bundle and elements of $\Grtil_{d,n}$ can be thought of as compatibly oriented Parseval frames with the same Gram matrix. There is a double cover $\Grtil_{d,n}\to \Grtil^u_{d,n}$ whose projection comes from forgetting the orientation. We denote $\E_{d,n}$ as the preimage of $\E^u_{d,n}$ under this covering map. Elements of $\E_{d,n}$ are equivalence classes of compatibly oriented \emph{equal norm} Parseval frames with the same Gram matrix, and one can again decide whether $F\in \Grtil_{d,n}$ is contained in $\E_{d,n}$ by examining the diagonal entries of the corresponding orthogonal projection matrix. 

\subsection{Obstruction Theory}\label{sec:obstruction_theory}
In this section we will give an overview of obstruction theory in the context of our problem. The background material here can be found in \cite[Ch 8]{Steenrod_Fibre_Bundles} or \cite[\S 4]{Hatcher_Algebraic_Topology}. Roughly speaking, we are concerned with problems of the following flavor. Let $(X,A)$ be a pair of topological spaces. Suppose we have a map $f:Y\to X$. Is it possible to homotope $f$ to have image in $A$?

Suppose that $Y$ is equipped with a (finite) CW structure. Explicitly, we have a filtration 
\[
\emptyset=Y^{(-1)}\subset Y^{(0)}\subset \cdots \subset Y^{(k)}=Y.
\]
Each stage is called a \emph{skeleton} and is obtained inductively from the previous one by attaching a collection of disks along their boundary. More precisely, $Y^{(0)}$ is a collection of points. Given $Y^{(i-1)}$, we form $Y^{(i)}$ by attaching a collection of $i$-disks by gluing the bounding $(i-1)$-sphere to $Y^{(i-1)}$. Once $Y$ has been constructed, we can regard each $i$-cell as a map of pairs $(\Delta^i,S^{i-1})\to (Y^{(i)},Y^{(i-1)})$.

Next, we recall the definition of relative homotopy groups. Given a pair of topological spaces $(X,A)$ we can consider the collection of homotopy classes of maps of triples $f:(\Delta^k,S^{k-1},\ast)\to (X,A,p)$, where $p\in A$. For  $k\geq2$, these homotopy classes form the relative homotopy group $\pi_k(X,A)$; for $k=1$ they form a pointed set. In the group case, the identity element consists of maps of triples that can be homotoped to have image in $A$. Another way to think about this is that an element of $\pi_k(X,A)$ is the same as an element $g\in \pi_{k-1}(A)$ along with a trivializing disk living in $X$. 

The spirit of obstruction theory is to try to understand the desired deformation problem one skeleton at a time. We will assume that $(X,A)$ is a CW pair, that $A$ is simply connected, and that the pair $(X,A)$ is $(d-1)$-connected in the sense that $\pi_k(X,A)=0$ for $k\leq d-1$, for $d > 2$; these assumptions are satisfied by the spaces we consider later in the paper. Consider the identity map of pairs $\Id:(X,A)\to(X,A)$. Throughout the following construction, we keep \(A\) pointwise fixed. By the connectivity hypothesis and the homotopy extension property for CW pairs, we can successively deform $\Id$, relative to $A$, so that its $(d-1)$-skeleton has image in $A$. Since we have homotoped the 0-skeleton into $A$, we can regard each $1$-cell as a map $(\Delta^1,S^0)\to (X,A)$. It hence defines an element of $\pi_1(X,A)$. Since $\pi_1(X,A)=0$ we can deform this cell to have image in $A$. Repeating this construction we can homotope $\Id$ so that its 1-skeleton has image in $A$. By repeatedly using this idea we can homotope $\Id$ so that its $(d-1)$-skeleton has image in $A$. When we start trying to deform the $d$-skeleton we encounter an honest obstruction. The $d$-cells of $X$ again give elements of $\pi_d(X,A)$, but our hypotheses no longer require this group to vanish, so these elements may be nonzero. In this way we have an assignment of an element of $\pi_d(X,A)$ to each $d$-cell of $X$. Any cell whose image is already in $A$ gives the trivial element of $\pi_d(X,A)$. In other words, we have constructed a relative cellular cochain with coefficients in $\pi_d(X,A)$. We denote this element as $o_{(X,A)}\in C^d(X,A;\pi_d(X,A))$. It is a classical fact that $o_{(X,A)}$ is a cocycle and hence defines an element $[o_{(X,A)}]\in H^d(X,A;\pi_d(X,A))$. This element is called the \emph{primary obstruction class}. Furthermore, $[o_{(X,A)}]=0$ if and only if $\Id$ can be deformed, relative to $A$, to a map carrying $X^{(d)}$ into $A$. Here vanishing of the cohomology class permits us to modify the deformation already chosen on the $(d-1)$-skeleton.

This procedure can be continued to produce further obstructions to deforming higher skeleta to have image in $A$, however, at this point the procedure becomes non-canonical because the obstructions at a particular stage depend on the choice of how lower skeleta were deformed into $A$. Fortunately, as we will see, in our cases of interest we will only need to consider primary obstructions. 

\begin{remark}
    It is possible to replicate the constructions above without the need to have a CW structure by using the technology of Postnikov towers, however, for those unfamiliar with these ideas it is often useful to frame obstruction theory in its classical CW context. Details about Postnikov towers and relations to obstruction theory can be found in \cite{Hatcher_Algebraic_Topology}.
\end{remark}

\section{Vector bundles}\label{sec:vector_bundles}

Let $\GL^+(d)$ be the group of orientation preserving automorphisms of $\RR^d$. This group admits a canonical faithful action on $\RR^d$. An orientable rank-$d$ vector bundle over a $k$-dimensional manifold is an $\RR^d$ bundle over $M$ with structure group $\GL^+(d)$. Material regarding vector bundles, classifying spaces, and tautological bundles is classical and can be found in \cite{Husemoller1994}.

Such bundles are classified by the \emph{stable oriented Grassmannian} $\Gr_{d}=\varinjlim_n\Gr_{d,n}$. Intuitively this space can be viewed as the space of oriented $d$-dimensional subspaces in a (countably) infinite dimensional real vector space. There is a canonical orientable rank-$d$ vector bundle, $\tau_d$, over $\Gr_d$. For each $P\in \Gr_d$ the fiber over $P$ is equal to $P$. This bundle is called the \emph{tautological bundle}. Given a smooth map $f:M\to \Gr_d$ we get a vector bundle $f^\ast(\tau_d)$ via pullback. It is a standard fact that every orientable rank-$d$ bundle over $M$ arises from this construction and homotopic maps give rise to isomorphic bundles. For this reason the (homotopy class of) map is called the \emph{classifying map} for the bundle. Each $\Gr_{d,n}$ is included in $\Gr_{d}$ and pulling back $\tau_d$ via the inclusion gives the tautological orientable rank-$d$ bundle, $\tau_{d,n}$ over $\Gr_{d,n}$. 

\subsection{Frame fields on vector bundles}\label{sec:frames_on_bundles}

Let $M$ be a smooth $k$-dimensional manifold and let $E\to M$ be an orientable rank-$d$ vector bundle. A collection $\{\sigma_1,\ldots,\sigma_n\}$ of sections of $E\to M$ is a \emph{frame field of size $n$} if, for each $p\in M$, $\{\sigma_1(p),\ldots \sigma_n(p)\}$ is a frame in the fiber $E_p$ over $p$. Next, assume that $E\to M$ is equipped with a smoothly varying fiberwise inner product. One can then impose the Parseval (resp.\ equal norm) condition pointwise and this gives rise to the notion of a frame field being Parseval (resp.\ equal norm).

In the orientable setting we have the following existence result for Parseval frame fields.

\begin{theorem}[\!\!{\cite[Thm 4.1]{BallasNeedhamShonkwiler2025}}]\label{thm:parseval_frames_exist}
Let $E\to M$ be an orientable, rank-$d$ vector bundle over a $k$-dimensional manifold $M$, equipped with a smoothly varying fiberwise inner product. For each $n\geq k+d$ there is a Parseval frame field of size $n$ on $E$.
\end{theorem}

Next, we show that there is a relationship between classifying maps and Parseval frame fields. Suppose one has a Parseval frame field $\mathcal{C}=\{\sigma_1,\ldots, \sigma_n\}$ on $E$. For each point $p\in M$, the frame field provides an analysis operator from $E_p\to \RR^n$. This map is always injective and hence its image defines a point in $\Gr^u_{d,n}$. There are 2 choices of orientation on the image and a point in $\Grtil_{d,n}$ is specified by selecting the orientation that agrees with the orientation on $E_p$ when pulled back via the analysis map. In other words, a Parseval frame field of size $n$ determines a map $f_{\mathcal{C}}:M\to \Grtil_{d,n}\subset \Grtil_d$. 

There is also a complementary construction going in the other direction. There are $n$ canonical sections of the tautological bundle $\tau_{d,n}$. For $1\leq i \leq n$ we define $s_i(P)$ to be the orthogonal projection of the $i$th standard basis vector of $\RR^n$ onto $P$. Next, suppose that $f:M\to \Grtil_d$ is a classifying map for $E$, and suppose further that $f$ can be homotoped to have image in $\Grtil_{d,n}$; we continue to denote the resulting map $M\to\Grtil_{d,n}$ by $f$. An orientation-preserving bundle isomorphism $T:E\to f^\ast\tau_{d,n}$ gives rise to an isomorphism $U = T(T^\ast T)^{-1/2}$ which is an isometry on each fiber, where the pullback carries the metric induced from $\mathbb R^n$. Using this isometry to identify the bundles, the pulled-back sections
\[
\mathcal C_f=\{U^{-1}(f^\ast s_1),\ldots,U^{-1}(f^\ast s_n)\}
\]
give a Parseval frame field of size $n$ on $E$. The fact that this frame is Parseval comes from the Han–Larson dilation theorem \cite{han2000frames} which says that Parseval frames come from orthogonally projecting an orthonormal basis onto a subspace.   
The following lemma describes how these constructions are related to one another.

\begin{lemma}\label{lem:classifying_map}
Let $E\to M$ be an oriented rank-$d$ vector bundle with a fiberwise inner product.
\begin{itemize}
 \item If $f:M\to \Grtil_{d,n}$ is a classifying map for $E$ then $\mathcal{C}_f$ is a Parseval frame field.
    \item Let $\mathcal{C}=\{\sigma_1,\ldots, \sigma_n\}$ be a Parseval frame field on $E$. Then the map $f_{\mathcal{C}}:M\to \Grtil_{d,n}$ described above is a classifying map for $E$. 
    \end{itemize}
\end{lemma}

\begin{proof}
The first point follows from the discussion preceding the lemma.

Since the sections and the fiberwise inner product are continuous, the corresponding analysis operators vary continuously, so $f=f_{\mathcal C}$ is continuous. Let $p\in M$ and $A_p:E_p\to \RR^n$ be the corresponding analysis operator. Since the frame is Parseval this gives an isometric embedding of $E_p$ into $\RR^n$. By construction of the pullback, $f^\ast(\tau_{d,n})_p=\operatorname{Im}(A_p)\subset \RR^n$. 

It follows that the analysis operator gives a fiberwise orientation-preserving isomorphism $A_p:E_p\to f^\ast(\tau_{d,n})_p$. Since the analysis operators vary continuously these maps piece together to give an isomorphism of vector bundles, and so $f^\ast(\tau_{d,n})\cong E$. 
\end{proof}

In light of Lemma \ref{lem:classifying_map} we will often speak about classifying maps coming from a certain Parseval frame. 

Recall from Section \ref{sec:background}, that elements of $\E_{d,n}$ correspond to $\SO(d)$-orbits of equal norm Parseval frames. Therefore, if one applies the above classifying map construction to an equal norm Parseval frame field, then one gets a classifying map $f$ whose image lives in $\E_{d,n}$. If the original Parseval frame is equal norm, then its corresponding map has image in $\E_{d,n}$. Conversely, given a classifying map $f:M\to \E_{d,n}$, the corresponding Parseval frame field $\mathcal{C}_f$ is equal norm.  The discussion above is summarized in the following proposition.

\begin{proposition}\label{prop:homotope_to_equal_norm}
The bundle $E\to M$ admits an equal norm Parseval frame field of size $n$ if and only if the classifying map $f:M\to \Grtil_{d}$ for $E$ can be homotoped to have image in $\E_{d,n}$.
\end{proposition}

\begin{proof}
    Homotoping a classifying map does not change the isomorphism class of its pullback, so if the classifying map can be homotoped into $\E_{d,n}$ then we can construct an equal norm Parseval frame of size $n$ on $E$. Conversely, if $E$ admits an equal norm Parseval frame then the construction above provides a classifying map $f:M\to \Gr_{d,n}$ whose image is contained in $\E_{d,n}$.
\end{proof}

\subsection{The Euler Class}\label{sec:eulerclass}

In this section, we describe the Euler class. The Thom class, Euler class, and their relationship via Poincar\'e duality are classical (see for instance \cite[Ch 11]{milnor1974characteristic} or \cite[Ch VI]{BottTu1982}). We will describe the class in two different ways. The first is motivated by thinking of the Euler class as an obstruction to finding non-zero sections of vector bundles and the second comes from applying Poincar\'e duality. Both of these descriptions and their relation will become particularly important in the proof of the main theorems in Sections \ref{sec:proof_of_main} and \ref{sec:complementary_bundles}. Consider an oriented rank-$d$ vector bundle $E$ over a $k$-dimensional manifold $M$. Consider the following sequence of pairs of fibrations
\[
(\RR^d,\RR^d\minus 0)\to (E,E\minus M)\to (M,M). 
\]

Here we are identifying $M$ with the zero section in $E$. Taking the corresponding long exact sequence, we find that $\pi_j(\RR^d,\RR^d\minus 0)\cong \pi_j(E,E\minus M)$ for each $j$. The pair $(\RR^d,\RR^d\minus 0)$ is well known to be $(d-1)$-connected and has $\pi_d(\RR^d,\RR^d\minus 0)\cong \ZZ$. A generator for $\pi_d(\RR^d,\RR^d\minus 0)$ corresponds to an orientation on $\RR^d$ and can be represented by a small disk centered at the origin. Similarly, a generator for $\pi_d(E,E\minus M)$ can be represented by a small disk centered at the origin in one's favorite fiber.

Since \(E\) is an oriented rank-\(d\) vector bundle, the Thom isomorphism gives
\[
H^d(E,E\setminus M;\mathbb Z)\cong H^0(M;\mathbb Z)\cong\mathbb Z.
\]
The Thom class $u_E$ is the distinguished generator whose restriction to each fiber is the positive generator of \(H^d(\mathbb R^d,\mathbb R^d\setminus\{0\};\mathbb Z)\). 

Given a section $\sigma:M\to E$, we can assume, by applying a small perturbation if necessary, that $\sigma$ is transverse to the zero section of $E$. We get a relative map $\sigma:(M,M\minus Z_\sigma)\to (E,E\minus M)$, where $Z_\sigma$ is the set of zeros of $\sigma$. By the transversality assumptions, $Z_\sigma$ is a smooth codimension $d$ submanifold. Pulling back gives $\sigma^\ast(u_E)\in H^d(M,M\minus Z_\sigma;\ZZ)$. 

The long exact sequence of the pair $(M,M\minus Z_\sigma)$ gives a natural map
\[
H^d(M,M\minus Z_\sigma;\ZZ)\to H^d(M;\ZZ).
\]
The image of $\sigma^\ast(u_E)$ under this map is the \emph{Euler class} $e(E)$. Any two choices of section are homotopic via a ``straight line'' homotopy, and so the corresponding absolute class is independent of the choice of section. This class provides an obstruction to the existence of \emph{any} nowhere zero section of $E$. 

There is a dual perspective that will be useful for our purposes. More detail regarding this perspective can be found in \cite{BottTu1982}. Let $\sigma:M\to E$ be a section transverse to the zero section, with zero set $Z_\sigma$. Since $M$ is closed and $\operatorname{rank}(E)=\dim M=d$, the set $Z_\sigma$ is finite. The orientations of $M$ and $E$ determine a signed zero-cycle
\[
[Z_\sigma]=\sum_{p\in Z_\sigma}\operatorname{ind}_p(\sigma)[p]\in H_0(Z_\sigma;\ZZ).
\]
Its relative Poincar\'e dual lies in
\[
H^d(M,M\minus Z_\sigma;\ZZ)
\cong \bigoplus_{p\in Z_\sigma}\ZZ
\]
and is characterized by its evaluation on a small positively oriented disk around each zero $p$, which equals $\operatorname{ind}_p(\sigma)$. Its image in $H^d(M;\ZZ)$ is the Euler class $e(E)$. The discussion above and the proposition below are standard (see \cite{milnor1974characteristic}, Chapter 11, or \cite{BottTu1982}, Chapter VI), but we include a proof for completeness because it is central to the argument. More generally, when the rank of \(E\) is smaller than the dimension of \(M\), the zero locus of a transverse section is a submanifold of codimension \(d\), and the same Thom-class description of its Poincaré dual applies.

\begin{proposition}\label{prop:poincare_dual_euler}
    Let $E\to M$ be an oriented rank-$d$ vector bundle and let $\sigma:M\to E$ be a section transverse to the zero section. Then
$$
PD([Z_\sigma])=\sigma^\ast(u_E)
\in H^d(M,M\minus Z_\sigma;\ZZ).
$$
Furthermore, for any $A\subset M\minus Z_\sigma$, the images of these classes in $H^d(M,A;\ZZ)$ agree. In particular, taking $A=\emptyset$ gives the Euler class $e(E)$.
\end{proposition}

\begin{proof}
    The first statement follows from the identification of Thom class with the relative Poincar\'e dual of the fundamental class of the zero locus described above. If $A\subset M\minus Z_{\sigma}$ then inclusion $(M,A)\to (M,M\minus Z_{\sigma})$ gives a natural map $H^d(M,M\minus Z_{\sigma})\to H^d(M,A)$. Under this map both the Thom and relative Poincar\'e dual map to their counterparts in $H^d(M,A)$. In particular, taking $A=\emptyset$ gives the Euler class. 
\end{proof}

We close the section by discussing some properties of the Euler class when $d$ is odd.

 Every vector bundle admits a non-trivial isomorphism given by applying the linear antipodal map fiberwise. When $d$ is even (resp.\ odd) this isomorphism is orientation preserving (resp.\ reversing). It follows that $e(E)=(-1)^de(E)$.  If $d$ is even this provides no information, but when $d$ is odd this tells us that the Euler class of $E$ is an element of $2$-torsion. We will use this fact in the form of the following lemma.

\begin{lemma}\label{lem:euler_is_torsion}
    Let $d$ be odd, and let $E\to M$ be an orientable rank-$d$ vector bundle over a closed orientable $d$-manifold. Then $e(E)=0$.
\end{lemma}

\begin{proof}
    By the discussion above, when $d$ is odd we have $e(E)\in H^d(M;\ZZ)\cong \ZZ$ is 2-torsion. This forces $e(E)=0$.
\end{proof}

\section{Topology of the Grassmannian/equal norm pair}\label{sec:GIT}

The main techniques of this section involve Geometric Invariant Theory (GIT). For background on the moment map and GIT techniques used in this
section, we refer to \cite{GeorgoulasRobbinSalamon}; see also
\cite{DolgachevInvariantTheory} for the algebraic GIT perspective. The main result of this section is similar to Theorem 2.8 of \cite{CaineNeedhamShonkwiler2025} (see also~\cite{caineNormalizingParsevalFrames2024}), and the ideas here were inspired by the work there. More specifically, we model the framework of the argument by constructing a certain energy function and showing that the corresponding basin of attraction for the negative gradient flow is the semistable set of a certain algebraic group action. 

We begin by proving some basic ``Schur–Horn'' type inequalities for elements of $\E_{d,n}$ that will become important later. Let $F\in \E_{d,n}$ and let $I\subset \{1,\ldots,n\}$. Let $V_I\subset F$ be the span of $\{P_Fe_i\}_{i\in I}$ and let $\ell$ denote its dimension.  It follows that $\tr(P_{V_I})=\sum_{i=1}^n\|P_{V_I}e_i\|^2=\ell$. For $i\in I$ we have $P_Fe_i\in V_I$, and so $\|P_{V_I}e_i\|^2=\|P_Fe_i\|^2$ for $i\in I$. It follows that
\[
\ell=\sum_{i=1}^n\|P_{V_I}e_i\|^2\geq \sum_{i\in I}\|P_Fe_i\|^2=\abs{I}\frac{d}{n}
\]
with the last equality coming because $F$ is equal norm. Morally speaking, this tells us that the span of a subset of vectors of size $\abs{I}$ coming from an equal norm Parseval frame of size $n$ in a $d$-dimensional space must have dimension at least $\abs{I}\frac{d}{n}$. As we will see, this inequality is intimately tied to the GIT arguments that we will use later.

Next, let $\ell\in \{0,\ldots, d-1\}$, and let $I\subset \{1,\ldots, n\}$. For $F\in \Grtil_{d,n}$, define $P_F^I$ to be the restriction of orthogonal projection onto $F$ to the coordinate subspace of $\RR^n$ whose coordinates are indexed by $I$. Then we define 
\[
Z_{\ell,I}=\{F\in \Grtil_{d,n}\mid \rank(P_F^I)\leq \ell\}
\]
and let $Z=\cup_{\ell,I}Z_{\ell,I}$, where the union is taken over all $\ell\in \{0,\ldots, d-1\}$ and all $I\subset \{1,\ldots,n\}$ with $\abs{I}>\frac{\ell n}{d}$. The discussion above shows that $\E_{d,n}\subset \Grtil_{d,n}\minus Z$. The main result of this section is the following.
\begin{theorem}\label{thm:homotop_equivalence}
    There is a strong deformation retraction from $\Grtil_{d,n}\minus Z$ onto $\E_{d,n}$. In particular, the inclusion $\E_{d,n}\hookrightarrow \Grtil_{d,n}\minus Z$ is a homotopy equivalence. 
\end{theorem}

The proof of Theorem \ref{thm:homotop_equivalence} is given by showing that negative gradient flow along a certain ``energy'' function $E$ provides a strong deformation retract.  

The relevant machinery from GIT will require us to work over the complex numbers. Let $\Grtil_{d,n}^\CC$ be the \emph{complex Grassmannian} consisting of $d$-dimensional subspaces in $\CC^n$. Inside $\Grtil_{d,n}^\CC$ we have the space $\E^\CC_{d,n}$ of equal norm points. The space $\Grtil_{d,n}^\CC$ can be equipped with a K\"ahler structure obtained by pulling back the standard Fubini–Study structure on projective space via the Pl\"ucker embedding $\Grtil_{d,n}^\CC\hookrightarrow \mathbb{P}(\wedge^d(\CC^n)).$ The compact torus group $K=(S^1)^n$ acts on $\Grtil_{d,n}^\CC$ via matrix multiplication by diagonal matrices with unit modulus entries. This $K$-action on $\Grtil_{d,n}^\CC$ is Hamiltonian, with moment map $\mu:\Grtil_{d,n}^\CC\to \RR^n$ given by $\mu(F)=\left(\abs{P_Fe_1}^2,\ldots,\abs{P_Fe_n}^2\right)$. Since $\tr(P_F)=d$ it follows that this map has image in the plane defined by the equation $\sum_ix_i=d$. If we let $c=(\frac{d}{n},\ldots,\frac{d}{n})$  and view $\Grtil^u_{d,n}$ as the real locus of $\Grtil^\CC_{d,n}$ then it follows that $\E^u_{d,n}=\mu^{-1}(c)\cap \Grtil^u_{d,n}$.  

Let $K^\CC$ be the complexification of $K$, which can be realized as diagonal matrices with non-zero complex entries. A point $F\in \Grtil_{d,n}^\CC$ is called \emph{c-semistable} if $\overline{K^\CC\cdot F}\cap \mu^{-1}(c)\neq \emptyset$ and unstable if it is not semistable. In words, semistable points are those whose $K^\CC$ orbit closure contains an equal norm point. There is a more concrete description of the set of semistable points, called the \emph{Hilbert–Mumford criterion}. Let $F\in \Grtil_{d,n}^\CC$. Using the Pl\"ucker embedding, we can write $F=[\sum_{B}w_B(F)e_B]$, where the sum is taken over all subsets $B\subset \{1,\ldots, n\}$ of size $d$. Let $S(F)=\{B\mid w_B(F)\neq 0\}$. This set indexes the non-zero coordinates of $F$ under the Pl\"ucker embedding and can alternately be described as those $B$ for which $\{P_Fe_i\}_{i\in B}$ is a basis for $F$. In this context, we let $1_B\in \RR^n$ be the vector with ones in the slots indexed by elements of $B$ and zero elsewhere. In the setting of torus actions one can state the Hilbert–Mumford condition concisely using the language of convex polytopes (see \cite[Theorem 9.2]{DolgachevInvariantTheory} and \cite[Theorem 3.4]{AmendolaKohnReichenbachSeigal}).

\begin{theorem}\label{thm:Hilber_Mumford}
With the terminology above, a point $F\in \Grtil_{d,n}^\CC$ is $c$-semistable if and only if the vector $c=(d/n,\ldots,d/n)$ is contained in the convex hull $\mathcal{C}$ of the set $\{1_B\mid B\in S(F)\}$. 
\end{theorem}

\begin{remark}
The point $c=(d/n,\ldots,d/n)$ need not be integral, whereas the
usual GIT formulation uses an integral character.  However, since
$nc=(d,\ldots,d)\in\ZZ^n$, we may clear denominators without changing
the Hilbert–Mumford inequalities.  We therefore continue to use the
more natural normalization $c$.
\end{remark}

Next, let $Z_{\ell,I}^\CC=\{F\in \Grtil_{d,n}^\CC\mid \rank(P_F^I)\leq \ell\}$, and let $Z^\CC=\cup_{\ell,I}Z_{\ell,I}^\CC$, where the union is taken over $\ell \in \{0,\ldots,d-1\}$ and $I\subset \{1,\ldots, n\}$ with $\abs{I}>\frac{\ell n}{d}$. With this notation we have the following.

\begin{theorem}\label{thm:complement_of_Z_is_semistable}
    The set of $c$-semistable points is exactly $\Grtil_{d,n}^\CC\minus Z^\CC$.
\end{theorem}

\begin{proof}
    We will show that points fail to satisfy the Hilbert–Mumford condition if and only if they are contained in $Z^\CC$.  First suppose that $F\in Z^\CC$, then $F\in Z_{\ell,I}^\CC$ for some appropriate choice of $\ell$ and $I$. We define an integer-valued vector $ a\in \RR^n$ so that $a_i=-(n-\abs{I})$ when $i\in I$ and $a_i=\abs{I}$ if $i\notin I$. It is straightforward to check that $\langle a,c\rangle=0$. Next, we will show that for each $B\in S(F)$ that $\langle a,1_B\rangle>0$. Thus the convex hull is on the positive side of a hyperplane that contains $c$, and hence does not contain $c$. 

    To see this observe that if $B\in S(F)$ then $\{P_Fe_i\}_{i\in B}$ is a basis for $F$, and is in particular linearly independent. Since the span of $\{P_Fe_i\}_{i\in I}$ has dimension at most $\ell$, it follows that $\abs{B\cap I}\leq \ell$. Computing we find that 
    \begin{align*}\langle a,1_B\rangle &=\sum_{i\in B}a_i
    =-(n-\abs{I})\abs{B\cap I}+\abs{I}(d-\abs{B\cap I})\\ &=\abs{I}d-n\abs{B\cap I}\geq \abs{I}d-n\ell>0
    \end{align*}

    On the other hand, suppose that $c$ is not in the desired convex hull $\mathcal{C}$, then there is a vector $a\in \RR^n$ so that $\langle a,1_B\rangle >\langle a,c\rangle$ for every $B\in S(F)$. By renumbering the coordinates we can assume without loss of generality that $a_1\leq a_2\leq\ldots\leq a_n$. Next, we construct a basis $B$ for $F$ by starting with the empty set and considering $v_i=P_Fe_i$ in order, adding each vector if and only if doing so preserves linear independence. Since $B$ is a basis for $F$ it follows that $B\in S(F)$. Furthermore, if we let $\ell_k$ be the dimension of the span of $\{v_1,\ldots, v_k\}$, then by construction we have $\abs{B\cap \{1,\ldots, k\}}=\ell_k$.

    Computing we find $
    \langle a, 1_B\rangle =\sum_{i=1}^na_ib_i$, where $b_i$ is the $i$th component of $1_B$. Set $\ell_0=0$. Then $b_i=\ell_i-\ell_{i-1}$, and so we have 
    \[
\sum_{i=1}^na_ib_i=\sum_{i=1}^na_i(\ell_i-\ell_{i-1})=\ell_na_n+\sum_{i=1}^{n-1}\ell_i(a_{i}-a_{i+1})=da_n+\sum_{i=1}^{n-1}\ell_i(a_{i}-a_{i+1}).
    \]

    On the other hand $\langle a,c\rangle =\frac{d}{n}\sum_{i=1}^na_i$. Writing $1=i-(i-1)$ gives 
    \[
    \frac{d}{n}\sum_{i=1}^na_i=\frac{d}{n}\left(\sum_{i=1}^na_i(i-(i-1))\right)=\frac{d}{n}\left(na_n+\sum_{i=1}^{n-1}i(a_i-a_{i+1})\right)=da_n+\sum_{i=1}^{n-1}\frac{di}{n}(a_i-a_{i+1})
    \]
    Since $\langle a,1_B\rangle>\langle a,c\rangle$ it follows that 
    $\sum_{i=1}^{n-1}(\ell_i-\frac{di}{n})(a_i-a_{i+1})>0$. Since each $a_i-a_{i+1}$ is non-positive by construction it follows that there must be some $k<n$ for which $\ell_k-\frac{dk}{n}<0$, and hence $k>\frac{\ell_kn}{d}$. It follows that if we let $I_k=\{1,\ldots, k\}$ then $F\in Z_{\ell_k,I_k}^\CC\subset Z^\CC$.
\end{proof}

\subsection{Kempf–Ness gradient flow}

We now introduce the energy function that we will flow along. Define $E:\Grtil_{d,n}^\CC\to \RR$ by $E(F)=\abs{\mu(F)}^2=\sum_{i=1}^n\abs{P_Fe_i}^4$. By applying Cauchy–Schwarz with the vectors $\mu(F)$ and $c$ we have 
\[
\frac{d^4}{n^2}=\langle \mu(F),c\rangle^2\leq \abs{\mu(F)}^2\abs{c}^2=\frac{d^2}{n}E(F)
\]

It follows that $\frac{d^2}{n}\leq E(F)$ with equality precisely when $\mu(F)=c$. Hence the energy is minimized exactly along the equal norm locus.

\begin{remark}
The Kempf–Ness theory is usually formulated so that the distinguished
level of the moment map is zero.  Here we instead work with the
unshifted moment map
\[
    \mu(F)=\bigl(\abs{P_Fe_1}^2,\ldots,\abs{P_Fe_n}^2\bigr)
\]
and the level
\[
    c=(d/n,\ldots,d/n).
\]
Since the image of $\mu$ lies in the affine hyperplane
$\sum_i x_i=d$, we have
\[
    \abs{\mu(F)}^2
    =
    \abs{\mu(F)-c}^2+\abs{c}^2.
\]
Thus the negative gradient flow of the unshifted energy
$E=\abs{\mu}^2$ agrees with the usual Kempf–Ness flow for the shifted
moment map $\mu-c$.  In particular, the standard results identifying
the semistable locus with the basin of attraction of the zero set
apply here with $\mu^{-1}(c)$ in place of $\mu^{-1}(0)$.
\end{remark}

By the Kempf--Ness theorem and the Kirwan--Ness
description of the gradient-flow stratification
\cite[Lemma~5.2.5(b) and Corollary~7.1.9]{Woodward2010},
the semistable locus is precisely the basin of
attraction of $\mu^{-1}(c)$ under the negative
gradient flow of $E$. Here we have used the fact
that the squared norms of $\mu$ and the centered
moment map $\mu-c$ differ by a constant, and
therefore have identical gradient flows.

By Atiyah's connectedness theorem
\cite[Theorem~1]{Atiyah1982}, the fibers of
the moment map on the complex Grassmannian
are connected. In particular, $\mu^{-1}(c)$
is connected. Since this is the global minimum
set of $E$, it is a connected component of
the critical set of $E$.

Finally, the moment map is proper because
the complex Grassmannian is compact.
Lerman's theorem
\cite[Theorem~1.1]{LermanMomentMap}
therefore implies that the negative gradient
flow defines a strong deformation retraction
of the semistable locus onto $\mu^{-1}(c)$.

Putting this all together we get the following:

\begin{theorem}\label{thm:complex_deformation_retraction}
    Negative gradient flow along the energy function $E$ gives a strong deformation retraction of $\Grtil_{d,n}^\CC\minus Z^\CC$ onto $\E^\CC_{d,n}$. 
\end{theorem}

We can now prove Theorem \ref{thm:homotop_equivalence}.

\begin{proof}[Proof of Theorem \ref{thm:homotop_equivalence}]
    Complex conjugation is an isometry under which the function $E$ is invariant. It follows that the gradient of $E$ is also conjugation invariant and hence the gradient flow preserves the real Grassmannian $\Grtil^u_{d,n}$. Here we have $\E^u_{d,n}=\E^\CC_{d,n}\cap \Grtil^u_{d,n}$ and define $Z^u=Z^\CC\cap \Grtil^u_{d,n}$. Hence $\Grtil^u_{d,n}\minus Z^u$ strong deformation retracts onto $\E^u_{d,n}$.  

    Let $H:\Grtil^u_{d,n}\minus Z^u\times [0,1]\to \Grtil^u_{d,n}\minus Z^u$ be the homotopy corresponding to this strong deformation retract and let $p:\Grtil_{d,n}\to \Grtil^u_{d,n}$ be the covering projection. Composing, we get a homotopy $H\circ p:\Grtil_{d,n}\minus Z\times [0,1]\to \Grtil^u_{d,n}\minus Z^u$. The identity is a lift of $H\circ p\vert_0$ and so we can lift $H\circ p$ through the covering map $p$ to get 
    $\tilde H:\Grtil_{d,n}\minus Z \times [0,1]\to \Grtil_{d,n}\minus Z$. Since the image of $H\circ p\vert_1$ is contained in $\E^u_{d,n}$ and $\E_{d,n}$ is the $p$-preimage of $\E^u_{d,n}$ it follows that the image of $\tilde H\vert_1$ is contained in $\E_{d,n}$. Finally, for each $x\in \E_{d,n}$, $H(p(x),t)=p(x)$, and so it follows that $\tilde H(x,t)=x$. Thus $\tilde H$ restricts to the identity along $\E_{d,n}$ throughout the homotopy, and we have the desired strong deformation retraction. 
    
\end{proof}

\section{Homotopy groups of the Grassmannian/equal norm pair}\label{sec:topology_of_pair}

In this section, we discuss the connectivity of the pair $(\Grtil_{d,n},\E_{d,n})$ and study its first non-trivial relative homotopy group in our cases of interest. By Theorem \ref{thm:homotop_equivalence} the inclusion of pairs $(\Grtil_{d,n},\E_{d,n})\hookrightarrow (\Grtil_{d,n},\Grtil_{d,n}\minus Z)$ induces isomorphisms on relative homotopy and homology groups. In order to describe the connectivity of this pair we will want to understand the codimension of the various strata of $Z$. 

There is a slightly different description of the $Z_{\ell,I}$. Let $E_{I^c}$ be the coordinate subspace indexed by the elements of the complement $I^c$. With this in mind it is easy to check that 
\[
Z_{\ell,I}=\{F\in \Grtil_{d,n}\mid \dim(F\cap E_{I^c})\geq d-\ell\}
\]

This set is the inverse image of a real Schubert variety under the orientation double cover (see \cite{FultonYoungTableaux} for background on Schubert varieties). Specifically, the codimension of $Z_{\ell,I}$ is computed in \cite[\S 9.4]{FultonYoungTableaux} and is given by
\begin{equation}\label{eq:codim_formula}
    (d-\ell)(\abs{I}-\ell)
\end{equation}
This gives the following codimension bound. 

\begin{theorem}\label{thm:possible_codimensions}
    If $d>2 $ and $n \geq 2d$, then each unstable rank locus $Z_{\ell,I}$ has codimension at least $d$. 
\end{theorem}
\begin{proof}
Write $m=|I|$. Since $Z_{\ell,I}$ is unstable
and $n\geq 2d$, we have
\[
    m>\frac{\ell n}{d}\geq 2\ell.
\]
Thus $m-\ell\geq \ell+1$, and the
codimension formula gives
\[
\begin{aligned}
    \operatorname{codim} Z_{\ell,I}
    &=(d-\ell)(m-\ell)\\
    &\geq (d-\ell)(\ell+1)\\
    &=d+\ell(d-1-\ell)\\
    &\geq d.
\end{aligned}
\]
\end{proof}

Just by knowing the codimension we can make some statements regarding the low-dimensional homotopy groups. Near a smooth point of a codimension-\(d\) component of \(Z\) that lies in no other component, $\Grtil_{d,n}$ looks locally like a neighborhood of the origin in $\RR^{d}\times \RR^{K-d}$, where $K=\dim(\Grtil_{d,n})$. A $d$-disk in the first factor will intersect $Z$ in exactly one point. Such a disk is called a \emph{meridional disk}. 

\begin{theorem}\label{thm:d-1_connected_plus_generators}
    If $d>2$  and $n\geq 2d$, then the pair $(\Grtil_{d,n},\E_{d,n})$ is $(d-1)$-connected and $\pi_d(\Grtil_{d,n},\E_{d,n})$ is generated by meridional disks around the various top-dimensional components of $Z$. 
\end{theorem}

\begin{remark}
Each point of a codimension-$d$ stratum determines a local meridian.
However, meridians at different points of the same stratum need not
be equal in $\pi_d(\Grtil_{d,n},\E_{d,n})$, and it is possible for these meridians to
satisfy nontrivial relations. Both phenomena occur when $n=2d+1$.
\end{remark}

\begin{proof}
First, recall that $\Grtil_{d,n}$ is simply connected for \(d,n-d\geq 2\). Since \(d>2\) and \(n\geq 2d\), this holds in our setting. An element of $\pi_k(\Grtil_{d,n},\Grtil_{d,n}\minus Z)$ is represented by a map of pairs $(\Delta^k,S^{k-1})\to (\Grtil_{d,n},\Grtil_{d,n}\minus Z)$. By relative transversality (see for instance \cite{GuilleminPollack1974}) we can homotope this map so that $\Delta^k$ is transverse to every stratum of a Whitney stratification of \(Z\) compatible with the rank loci. Each stratum has codimension at least $d$ in $\Grtil_{d,n}$, and so if $k<d$ then $\Delta^k$ will be disjoint from $Z$ for dimensional reasons. Thus $(\Delta^k,S^{k-1})\to (\Grtil_{d,n},\Grtil_{d,n}\minus Z)$ can be homotoped to have image in $\Grtil_{d,n}\minus Z$ and so the corresponding class in $\pi_k(\Grtil_{d,n},\Grtil_{d,n}\minus Z)$ is trivial. 

By Hurewicz, we have that $\pi_d(\Grtil_{d,n},\E_{d,n})\cong H_d(\Grtil_{d,n},\E_{d,n}).$ Applying Poincar\'e–Lefschetz duality gives $H_d(\Grtil_{d,n},\E_{d,n})\cong H_d(\Grtil_{d,n},\Grtil_{d,n}\minus Z)\cong H^{K-d}(Z)$, where $K=\dim(\Grtil_{d,n})$. Since $Z$ has codimension $d$ this is the top dimensional cohomology of $Z$ and we see that $H_d(\Grtil_{d,n},\E_{d,n})$ is generated by meridional disks that are Poincar\'e dual to the top dimensional components of $Z$. 
\end{proof}

\subsection{The top dimensional pieces}

Theorem \ref{thm:d-1_connected_plus_generators} tells us that the generators of the first non-trivial homotopy group come from top dimensional pieces of $Z$. Our next goal is to enumerate these pieces. 

By examining the equality cases in the proof of Theorem \(\ref{thm:possible_codimensions}\), we see that codimension \(d\) occurs only when \((\ell,|I|)=(0,1)\) or \((d-1,2d-1)\). The first case gives the \(n\) components \(Z_i:=Z_{0,\{i\}}\), corresponding to \(d\)-planes orthogonal to the \(i\)th coordinate vector.

In the second case $d\leq \abs{I}$, and so $\ell=d-1$. In order to have codimension $d$ it must be the case that $\abs{I}=2d-1$. When $n=2d$ there are again $n$ of these indexed by the $n$ complements $J_i:=\{1,\ldots,n\}\minus \{i\}$. These are $N_i:=Z_{d-1,J_i}$. The $N_i$ also have a very simple description. Precisely, $N_i$ corresponds to $d$-planes containing $e_i$. To see this observe that membership $F\in N_i$ means that the projections of $\{e_j\}_{j\neq i}$ onto $F$ fail to span $F$. Hence there is a non-zero $v\in F$ with $\langle v,P_F(e_j)\rangle=0$ for $P_F$ the orthogonal projection operator and $j\neq i$. Computing we find 

\[
\langle v,e_j\rangle=\langle P_F(v),e_j\rangle = \langle v,P_F(e_j)\rangle=0
\]

Hence $v$ is orthogonal to $\{e_j\}_{j\neq i}$ and so $v$ is a multiple of $e_i$. 

When $n=2d+1$ there are $\binom{n}{2}$ such components indexed by $J_{i,j}:=\{1,\ldots,n\}\minus\{i,j\}$ and defined as $D_{i,j}:=Z_{d-1,J_{i,j}}$. $D_{i,j}$ can similarly be described as $d$-planes that intersect the $ij$-coordinate plane non-trivially.

If $n\geq 2d+2$, then the instability inequality and the fact that $d>2$ imply
\[
d\abs{I}=2d^2-d < 2d^2-2=(2d+2)(d-1)\leq n\ell.
\]
But then $\abs{I} < \frac{n\ell}{d}$, so these rank conditions do not define unstable strata when $n\geq 2d+2$, and hence do not need to be considered. We summarize this discussion with the following lemma.

\begin{lemma}\label{lem:enumerate_top_dim_pieces}
    Let $d>2$ and $n\geq 2d$.
    \begin{itemize}
        \item Each $Z_i$ is an unstable codimension-$d$ component of $Z$.
        \item If $n\geq 2d+2$, there are no additional codimension-$d$ components.
        \item If $n=2d+1$, each $D_{i,j}$ is also a codimension-$d$ component, for $1\leq i<j\leq n$.
        \item If $n=2d$, each $N_i$ is also a codimension-$d$ component.
    \end{itemize}
\end{lemma}

\subsection{Relations among the generators}\label{sec:relations}

In this section we discuss possible relations amongst the generators of $\pi_d(\Grtil_{d,n},\E_{d,n})$ coming from Theorem \ref{thm:d-1_connected_plus_generators}. Relations amongst pairs of meridional generators can occur if there are strata that intersect a particular component. The goal of this section is to understand precisely how this can happen. We restrict the following analysis to $n=2d$ and $n\geq 2d+2$.

\begin{proposition}\label{prop:codimension_of_intersections1}
    Assume $d>2$ and $2d+1\neq n\geq 2d$. Let $S$ be one of the codimension-$d$ strata of $Z$ and let $Z_{\ell,I}$ be one of the other unstable strata. The codimension of $Z_{\ell,I}\cap S$ in $S$ is at least 2. 
\end{proposition}

\begin{proof}
    The proof breaks into 2 cases depending on the type of $S$. 
  
    First, suppose $S=Z_i$. Each point in $Z_i$ is a $d$-plane orthogonal to $e_i$. If $i\notin I$ then none of the projections of the vectors indexed by $I$ are forced to be trivial and so we are asking for $\abs{I}$ vectors to span a space of dimension at most $\ell$, and hence the codimension $c_{\ell,I}:=(\abs{I}-\ell)(d-\ell)\geq d$. When $i\in I$ the projection of one of the vectors is forced to be trivial and so we are asking $\abs{I}-1$ vectors to span a space of dimension at most $\ell$ and so the codimension is $(\abs{I}-1-\ell)(d-\ell)=c_{\ell,I}-(d-\ell)$. 

    Since $n\geq 2d$, the instability condition implies
$\abs{I}\geq 2\ell+1$. Thus, when $i\in I$ and
$\ell\geq 1$, the codimension is at least
\[
(\abs{I}-1-\ell)(d-\ell)
\geq \ell(d-\ell)
\geq d-1\geq 2.
\]
When $\ell=0$, the case $I=\{i\}$ is excluded,
and otherwise the codimension is at least $d$.

    This proves the required bound for $S=Z_i$. 

    We now move on to the second case, where $S=N_i$, which implies $n=2d$. Points of $S$ are $d$-planes that contain $e_i$. Let $P$ be the orthogonal projection onto this $d$-plane. If $j\neq i$, then 
    \[
    \langle e_i,P(e_j)\rangle=\langle P(e_i),e_j\rangle=\langle e_i,e_j\rangle=0.
    \]
    
    If $i\in I$, then being in $Z_{\ell,I}$ is asking for $\abs{I}-1$ vectors to lie in a subspace of dimension at most $\ell-1$ in the $(d-1)$-dimensional orthogonal complement to $e_i$, and the codimension is 
    \[(\abs{I}-1-(\ell-1))((d-1)-(\ell-1))=c_{\ell,I}\geq d > 2.\]
    
    If $i\notin I$ then we are asking for $\abs{I}$ vectors to span a space of dimension at most $\ell$ in the $(d-1)$-dimensional complement of $e_i$. Thus the codimension is $(\abs{I}-\ell)(d-1-\ell)=c_{\ell,I}-(\abs{I}-\ell)$. For $0\leq\ell\leq d-2$, the instability
condition gives $\abs{I}\geq 2\ell+1$, and hence
\[
(\abs{I}-\ell)(d-1-\ell)
\geq(\ell+1)(d-1-\ell)
\geq d-1\geq 2.
\]
If $\ell=d-1$, instability forces
$\abs{I}\geq 2d-1$. Since $i\notin I$,
we must have
$I=\{1,\ldots,2d\}\minus\{i\}$,
which defines $N_i$ itself and is excluded.

    \end{proof}

    Using this result we can compute $\pi_d(\Grtil_{d,n},\Grtil_{d,n}\minus Z)$ when $n=2d$ or $n\geq 2d+2$. 

\begin{proposition}\label{prop:homotopy_of_stiefel_pair}
    Let $d>2$.
    \begin{itemize}
        
        \item If $n\geq 2d+2$, then $\pi_d(\Grtil_{d,n},\Grtil_{d,n}\minus Z)\cong \ZZ^n$ with generators given by meridional disks around each of the $Z_i$.
        \item If $n= 2d$, then $\pi_d(\Grtil_{d,n},\Grtil_{d,n}\minus Z)\cong \ZZ^{2n}$ with generators given by meridional disks around each of the $Z_i$ and the $N_i$.
    \end{itemize}
\end{proposition}

\begin{proof}

    Let $q=K-d$ where $K=\dim(\Grtil_{d,n})$. By Theorem \ref{thm:d-1_connected_plus_generators}, we have shown that the pair $(\Grtil_{d,n},\Grtil_{d,n}\minus Z)$ is $(d-1)$-connected. By relative Hurewicz it follows that $\pi_d(\Grtil_{d,n},\Grtil_{d,n}\minus Z)\cong H_d(\Grtil_{d,n},\Grtil_{d,n}\minus Z)$. Applying Alexander–Lefschetz duality it follows that $H_d(\Grtil_{d,n},\Grtil_{d,n}\minus Z)\cong H^{q}(Z)$, 

Let $Z^+$ be the union of the top dimensional pieces of $Z$ and $Z^-$ the union of the remaining pieces. We can apply Mayer-Vietoris to the decomposition $Z=Z^+\cup Z^-$.

\[
H^{q-1}(Z^+\cap Z^-)
\longrightarrow H^q(Z)
\longrightarrow H^q(Z^+)\oplus H^q(Z^-)
\longrightarrow H^q(Z^+\cap Z^-).
\]

However, we know that $\dim Z^-\leq q-1$ and $\dim(Z^+\cap Z^-)\leq q-2$. This means that the outer 2 groups and $H^q(Z^-)$ all vanish and we find that $H^q(Z)\cong H^q(Z^+)$.

Any nonempty intersection of pairs of distinct components of $Z^+$
is an oriented Grassmannian. When $n\geq 2d+2$, the
components are the $Z_i$, and
$Z_i\cap Z_j\cong\Grtil_{d,n-2}$ for $i\neq j$,
which has codimension $d$ in $Z_i$.

When $n=2d$, the components are the $Z_i$ and $N_i$.
Intersections of two distinct components of the same type
have codimension $d$, while $Z_i\cap N_j$ has codimension
$d-1$ for $i\neq j$ and is empty otherwise.

Since $d>2$, all intersections of distinct components
have codimension at least two. Repeated application of
a Mayer-Vietoris argument similar to the one above therefore gives
\[
H^q(Z^+)\cong\bigoplus_S H^q(S),
\]
where $S$ ranges over the components of $Z^+$.
Each $S$ is a closed, connected, oriented manifold, so
each summand is isomorphic to $\ZZ$ and the result follows by applying Lemma \ref{lem:enumerate_top_dim_pieces}.

\end{proof}

In the $n=2d+1$ case we are not able to compute $\pi_d(\Grtil_{d,n},\Grtil_{d,n}\minus Z)$. However, we are able to deduce some relations inside this group.

The main idea is that the generators of $\pi_d(\Grtil_{d,n},\Grtil_{d,n}\minus Z)$ come from top dimensional pieces. One source of relations occurs when two pieces intersect and their intersection has codimension 1 inside the respective pieces. When $n\geq 2d$ with $n\neq 2d+1$ we were able to rule out such intersections. When $n=2d+1$ these types of intersections can occur. More precisely $N_j\subset D_{ij}\cap D_{jk}$ and has codimension 1 inside $D_{ij}$ and $D_{jk}$.

We start by describing a convenient coordinate system for the normal bundle to $N_i$ at a ``generic'' point. In this case, define $N_i=\{F\in\Grtil_{d,2d+1}\mid e_i\in F\}$, and let $F\in N_i$ be generic in the sense that if we orthogonally decompose as $F=W\oplus \RR e_i$ then all of the Pl\"ucker coordinates of $W$ are non-zero. This condition ensures that a sufficiently small disk centered $F$ will be disjoint from all components of $Z$ other than $N_i$ and $D_{ij}$ with $j\neq i$.

Next, we construct coordinates for a normal disk to $N_i$ in $\Grtil_{d,2d+1}$ at $F$. Since $e_i\in F$ we can write $F=W\oplus \RR e_i$, where $W$ is the orthogonal complement of $e_i$ in $F$. A small neighborhood of the origin in the $(d+1)-$dimensional space $F^\perp$ provides a normal disk via the map $u\mapsto W\oplus \RR(e_i+u)$.

First observe that if $G\in \Grtil_{d,2d+1}\minus (N_i\cup N_j)$ we can define nowhere zero sections $\tau_i,\tau_j:$ $\Grtil_{d,2d+1}\minus (N_i\cup N_j)$ $\to \tau_{d,2d+1}^\perp$ where $\tau_i(G)$ is the normalized orthogonal projection of $e_i$ onto $G^\perp$. We can normalize these sections to have unit length everywhere and we will. When $G\in D_{ij}\minus (N_i\cup N_j)$ then we can find $a,b\neq 0$ so that $ae_i+be_j\in G$. Taking orthogonal projections and normalizing tells us that $\tau_i(G)=\pm \tau_j(G)$. It follows that $D_{ij}\minus (N_i\cup N_j)$ consists of two components depending on whether $\tau_i=\tau_j$ or $\tau_i=-\tau_j$. We henceforth refer to these components as $D_{ij}^+$ and $D_{ij}^-$, respectively.

Next, we analyze how $D_{ij}$ intersects this disk near $F$. Define $p_j\in F^\perp$ to be the orthogonal projection of $e_j$ onto $F^\perp$. Since $F$ is generic we have $p_j\neq 0$. Suppose that $G\in D_{ij}$ lies in this disk, then $G=W+\RR(e_i+u)$ for some small $u\in F^\perp$. On the other hand, since $G\in D_{ij}$ we see that $G$ contains a vector of the form $v=ae_i+be_j$. Since $e_j\notin G$ there is no harm in assuming that $a=1$. Thus we find that $v=e_i+be_j$. On the other hand $v=w+c(e_i+u)$ for some $w\in W$. There is an orthogonal decomposition $\RR^{2d+1}=F^\perp \oplus W\oplus \RR e_i$ with respect to which the second description of $v$ is written. It follows that $c=1$ and $u=P_{F^\perp}(v)=bp_j$. Conversely, if $u=bp_j$ for some sufficiently small $b$ then we can argue similarly to see that the plane coming from $u$ lies in $D_{ij}$. It follows that the intersection of $D_{ij}$ with the normal disk corresponds precisely to (a portion of) the line in $F^\perp$ spanned by $p_j$.

More precisely, suppose that $u=tp_j$ and $G_u=W\oplus \RR(e_i+u)$. Using the splitting $\RR^{2d+1}=F^\perp\oplus W\oplus \RR e_i$ we see that we can write $e_j=p_j+w_j$, where $w_j\in W$ (here we use the fact that $e_i$ and $e_j$ are orthogonal). It follows that
\[v=e_i+te_j=tw_j+e_i+tp_j=tw_j+e_i+u\]
and hence $v\in G_u$. Taking orthogonal projections onto $G_u^\perp$ we find that $G_u\in D_{ij}^+$ (resp.\ $D_{ij}^-$) if and only if $t<0$
(resp. $t>0$).

Repeating for each $j\neq i$ we see a local picture of a normal $(d+1)$-disk to $N_i$ at $F$ whose intersections with the $D_{ij}$ give $2d$ diameters. Our genericity condition also ensures that these diameters are all distinct. The diameters intersect the boundary $d$-sphere in $4d$ points. These $4d$ points correspond to $D_{ij}^\pm$ for each $j\neq i$. Small disks centered at these points give rise to elements $\{d_{i,j}^\pm\}_{j\neq i}$ of $\pi_d(\Grtil_{d,2d+1},\Grtil_{d,2d+1}\minus Z)$, here the disks are oriented using the orientation of the boundary of the normal.

By relative Hurewicz we have
\[\pi_d(\Grtil_{d,2d+1},\Grtil_{d,2d+1}\minus Z)\cong H_d(\Grtil_{d,2d+1},\Grtil_{d,2d+1}\minus Z).\]

The complement of these $4d$ disks leaves a subset of the complement of $Z$. Consequently, in relative homology the boundary sphere represents the sum of the meridional disks, oriented by the sphere. Since this sphere bounds the normal disk, the sum of these meridians vanishes. More precisely we get a relation $R_i=\sum_{j\neq i}(d_{i,j}^+ + d_{i,j}^-)=0$. 

\section{The stable case}\label{sec:proof_of_main}

Let $n\geq 2d+2$ and $d>2$, and let $E\to M$ be an orientable rank-$d$ vector bundle over a $d$-manifold $M$. Given Proposition \ref{prop:homotope_to_equal_norm}, in order to produce an equal norm Parseval frame of size $n$ we need to homotope a classifying map $f:M\to \Grtil_d$ to have image in $\E_{d,n}$. Since $n > 2d$, we can apply Theorem \ref{thm:parseval_frames_exist} to find a Parseval frame field $\{\sigma_1,\ldots, \sigma_{n}\}$. This is equivalent to deforming the classifying map $f$ to have image in $\Grtil_{d,n} \subset \Gr_d$.

We know from Theorem \ref{thm:d-1_connected_plus_generators} that $(\Grtil_{d,n},\E_{d,n})$ is $(d-1)$-connected. Thus it is possible to further deform $f$ so that it maps the $(d-1)$-skeleton of $M$ into $\E_{d,n}$. By Proposition \ref{prop:homotopy_of_stiefel_pair}, we also know that $\pi_d(\Grtil_{d,n},\E_{d,n})\cong \ZZ^n$. Hence,
by the discussion from Section \ref{sec:obstruction_theory} we have a primary obstruction class $o_{ENF}(E)\in H^d(M,\ZZ^n)$ that vanishes precisely when we can homotope the classifying map so the $d$-skeleton (and hence all of $M$) has image in $\E_{d,n}$. This obstruction can be cast in a more universal language using $\Grtil_{d,n}$. Let $o_{ENF}\in H^d(\Grtil_{d,n},\E_{d,n};\pi_d(\Grtil_{d,n},\E_{d,n}))$ be the primary obstruction to deforming the $d$-skeleton of $\Grtil_{d,n}$ to have image in $\E_{d,n}$. If $f:M\to \Grtil_{d,n}$ is a classifying map for a bundle $E\to M$, then $o_{ENF}(E)=f^\ast(o_{ENF}).$

Our next goal is to recast the obstruction in terms of the Euler class. We have seen 
\[
\pi_d(\Grtil_{d,n},\E_{d,n})\cong \pi_d(\Grtil_{d,n},\Grtil_{d,n}\minus Z)\cong \ZZ^n,
\]
with generators coming from meridional disks $\mu_i$ around the top dimensional components $Z_i$. Using these coordinates and projecting, we get maps 
\[
p_i:H^d(\Grtil_{d,n},\E_{d,n};\pi_d(\Grtil_{d,n},\E_{d,n}))\to H^d(\Grtil_{d,n},\E_{d,n};\ZZ).
\]
We also write $p_i$ for the induced coefficient projections $H^d(M;\ZZ^n)\to H^d(M;\ZZ)$. Orient each meridional disk using the orientation of the tautological bundle via the derivative of its defining section. The next lemma describes how $p_i(o_{ENF}(E))$ is related to the Euler class.

 \begin{lemma}\label{lem:relation_to_Euler}
     Let $n$, $d$, and $E$ be as above. Then for each $i$,
     \begin{itemize}
         \item $p_i(o_{ENF}(E))=e(E)$
         \item $\bigcap_i \ker p_i=0$.
     \end{itemize}
 \end{lemma}

 \begin{proof}

For $1\leq i\leq n$ there is a section  $\sigma_i:\Grtil_{d,n}\to \tau_{d,n}$ of the tautological bundle defined as follows. For each $P\in \Grtil_{d,n}$ we define $\sigma_i(P)$ to be the orthogonal projection of the $i$th standard basis vector of $\RR^n$ onto $P$. If we let $Z_{\sigma_i}$ be the zero set of this section, we know that it is closely related to the $i$th top-dimensional component $Z_i$ of $Z$. In particular $Z_i=Z_{\sigma_i}$. Each $Z_{\sigma_i}$ is a $(K-d)$-dimensional manifold where $K=\dim(\Grtil_{d,n})$.  In fact, $Z_{\sigma_i}$ can be identified with a lower dimensional Grassmannian $\Grtil_{d,n-1}$.

The section $\sigma_i:\Grtil_{d,n}\to \tau_{d,n}$ is transverse to the zero section and $\E_{d,n}\subset \Grtil_{d,n}\minus Z_{\sigma_i}$. Applying Proposition \ref{prop:poincare_dual_euler} with $A=\E_{d,n}$, we find that $PD([Z_{\sigma_i}])=\sigma_i^\ast(u_{\tau_{d,n}})\in H^d(\Grtil_{d,n},\E_{d,n})$. By Hurewicz $\pi_d(\Grtil_{d,n},\E_{d,n})\cong H_d(\Grtil_{d,n},\E_{d,n})$, with Proposition \ref{prop:homotopy_of_stiefel_pair} furnishing generators coming from meridional disks $[\mu_i]$ around the $Z_{\sigma_i}$. By applying universal coefficients along with $(d-1)$-connectivity of the pair we see that classes in $H^d(\Grtil_{d,n},\E_{d,n})$ are completely determined by their Kronecker pairing with the $[\mu_i]$. By construction, we have $\langle p_i(o_{ENF}),[\mu_j]\rangle =\delta_{ij}$. On the other hand, since $[\mu_i]$ is a meridional disk around $Z_{\sigma_i}$ we have $\langle PD([Z_{\sigma_i}]),[\mu_j]\rangle=\delta_{ij}$ as well. Hence $p_i(o_{ENF})=\sigma_i^\ast(u_{\tau_{d,n}})$.

Next, let $f:M\to \Grtil_{d,n}$ be a classifying map for $E$. Then we have

\[
p_i(o_{ENF}(E))=p_i(f^\ast o_{ENF})=f^\ast(p_i(o_{ENF}))=f^\ast(\sigma_i^\ast(u_{\tau_{d,n}}))=e(E).
\]

Finally, by universal coefficients we have 
\[H^d(\Grtil_{d,n},\E_{d,n};\pi_d(\Grtil_{d,n},\E_{d,n}))\cong \Hom(H_d(\Grtil_{d,n},\E_{d,n}),\ZZ^n)\cong \oplus_{i=1}^n\Hom(H_d(\Grtil_{d,n},\E_{d,n}),\ZZ).\]

The projections $p_i$ are essentially just the components with respect to the splitting above, so $\bigcap_i \ker p_i=0$ follows immediately.  \end{proof}

\begin{proof}[Proof of Theorem \ref{thm:main_thm}]
    One direction of the proof is straightforward. Suppose that one is able to construct an equal norm Parseval frame field of size $n$ consisting of sections $\sigma_1,\ldots, \sigma_n$. Each $\sigma_i$ gives a section of $E$ that is nowhere zero. Thus the Euler class of $E$ must be trivial. 

     On the other hand, suppose that $e(E)=0$. 
    Let $o_{ENF}(E)$ be the obstruction to homotoping the classifying map for $E$ to have image in $\E_{d,n}$. By Lemma $\ref{lem:relation_to_Euler}$, we have that $p_i(o_{ENF}(E))=e(E)=0$. Thus, for each $i$, $o_{ENF}(E)\in \ker p_i$. Again applying Lemma \ref{lem:relation_to_Euler} shows that $o_{ENF}(E)=0$. Hence it is possible to deform the classifying map into $\E_{d,n}$, and hence we can find an equal norm Parseval frame of size $n$.
\end{proof}

\section{Complementary bundles and the critical dimensions}\label{sec:complementary_bundles}

In this section we discuss the sharpness of Theorem \ref{thm:main_thm} and Theorem \ref{thm:parseval_frames_exist}. A primary theme of this section is that a Parseval frame provides an isometric embedding into the trivial bundle and that characteristic classes of the orthogonal complement bundle provide obstructions to constructing equal norm Parseval frames. 

\subsection{Parseval frame fields}
We begin by letting $E\to M$ be a rank-$d$ oriented vector bundle with fiberwise inner products over a $k$-dimensional manifold. By Theorem \ref{thm:parseval_frames_exist} there is a Parseval frame field $\{\sigma_1,\ldots, \sigma_n\}$, where $n=k+d$. Using analysis operators, we get an embedding $E \to M\times \RR^n$ which is a bundle map and an isometric embedding on each fiber. Next, let $E^\perp$ be the orthogonal bundle, which is a rank-$k$ vector bundle over $M$. That is, the fiber $E^\perp_p$ is the orthogonal complement of $E_p$ in $\RR^n$, under the above bundle map embedding. 

There is a more universal description of $E^\perp$ that makes its dependence on the choice of frame more precise. For each $\tau_{d,n}\to \Grtil_{d,n}$ there is a corresponding rank-$(n-d)$ bundle $\tau_{d,n}^\perp \to \Grtil_{d,n}$. These two bundles sum to give the trivial $\RR^n$ bundle over $\Grtil_{d,n}$. Each frame of size $n$ gives a classifying map $f:M\to \Grtil_{d,n}$ and $E^\perp=f^\ast(\tau_{d,n}^\perp)$.

\begin{remark}\label{rem:perp_depends_on_frame}
It is important to note that $E^\perp$ is \emph{not} an invariant of $E$, rather it  depends on the frame giving rise to the isometric embedding into $M\times \RR^n$. In general, it is possible that different Parseval frames will give rise to non-isomorphic complements with different Euler classes. 

\end{remark}

The first problem we will address is the problem of \emph{shortening} a Parseval frame field. Given a Parseval frame field $\mathcal{C}=\{\sigma_1,\ldots, \sigma_n\}$ on $E\to M$ a \emph{shortening of $\mathcal{C}$} is a family $\mathcal{C}_t=\{\sigma_1^t,\ldots, \sigma_n^t\}$ of Parseval frame fields so that $\mathcal{C}_0=\mathcal{C}$ and $\sigma^1_n$ is the zero section. If we remove this zero section from $\mathcal{C}_1$ then we obtain a Parseval frame field of size $n-1$.

Shortening also admits a convenient description at the level of classifying maps.

\begin{proposition}\label{prop:shortenable_classifying_map}
    A Parseval frame field of size $n$ on $E\to M$ is shortenable if and only if the corresponding classifying map $f:M\to \Grtil_{d,n}$ can be homotoped to have image in $\Grtil_{d,n-1}$.
\end{proposition}

\begin{proof}

If a Parseval frame contains the zero vector as its last vector then the corresponding analysis operator automatically has image in the coordinate hyperplane $\RR^{n-1}\hookrightarrow \RR^{n}$. Thus for a shortenable Parseval frame field its corresponding classifying map can be deformed to have image in $\Grtil_{d,n-1}\subset \Grtil_{d,n}$. Conversely, suppose there is a homotopy $f_t$ of our original classifying map so that the image of $f_1$ lies in $\Grtil_{d,n-1}$. We can use this map to pull back $\tau_{d,n}$ to give a rank-$d$ bundle $\mathcal{E}$ over $M\times I$. By homotopy invariance for Euclidean vector bundles we have that $pr_M^\ast E\cong \mathcal{E}$. More precisely, we have a family of isometries $\phi_{(p,t)}:E_p\to \mathcal{E}_{(p,t)}\subset \RR^n$. These isomorphisms can be chosen so that $\phi_{(p,0)}=Id$. For each $p\in M$ we can regard $\phi_{(p,t)}$ as an analysis operator. These analysis operators are all isometric embeddings, and so this family corresponds to a family $\mathcal{C}_t$ of Parseval frame fields. The image of the analysis operators for $\mathcal{C}_1$ are all contained in $\RR^{n-1}\hookrightarrow \RR^n$, which means that the last section $\sigma_n^1\in \mathcal{C}_1$ has the property that $\sigma^1_n(p)$ is orthogonal to each vector in $E_p$. Thus $\sigma^1_n$ is the zero section and so the original frame is shortenable. 
\end{proof}

We now analyze when a Parseval frame field can be shortened. The following proposition describes how the Euler class of $E^\perp$ is the only obstruction to finding a shortening.

\begin{proposition}\label{prop:euler_obstruction_for_parseval_frames}
Let $E$ be as above with $d>2$. A Parseval frame field of size $d+k$ can be shortened if and only if the corresponding complement satisfies $e(E^\perp)=0$.

\end{proposition}

\begin{proof}
For the forward direction, let $f:M\to \Grtil_{d,d+k}$ be the classifying map associated to the Parseval frame. Since the Parseval frame can be shortened, there is a homotopy $f_t:M\to \Grtil_{d,d+k}$ so that the image of $f_1$ lies in $\Grtil_{d,d+k-1}$. Let $E_1^\perp$ be the complement associated to $f_1$, then by homotopy invariance we see that $E^\perp \cong E^\perp_1$. By construction, the last coordinate vector $e_{d+k}$ of $\RR^{d+k}$ is orthogonal to every $f_1(p)$, for $p\in M$, and so the constant vector field $e_{d+k}$ pulls back to give a nowhere zero vector field on $E_1^\perp$, and so $0=e(E_1^\perp)=e(E^\perp)$. 

Here the endpoint $f_1$ is regarded as a map into $\Grtil_{d,d+k}$ via the inclusion $\Grtil_{d,d+k-1}\hookrightarrow \Grtil_{d,d+k}$.

Conversely, suppose we can find a Parseval frame field on $E$ of size $d+k$ such that the corresponding bundle $E^\perp$ has $e(E^\perp)=0$. The triviality of the Euler class means we can find a nowhere zero vector field $\tau:M\to E^\perp$ on $E^\perp$. Without loss of generality we can assume that $\tau$ has been normalized to have unit length. $E^\perp$ is embedded in $M\times \RR^{d+k}$ and so we can regard $\tau$ as a map from $M$ to $S^{d+k-1}$. Let $F=\langle \tau \rangle^\perp$. Since $\tau(p)\in E_p^\perp$ it follows that $E_p\subset F_p\subset \epsilon^{d+k}$.  Our next goal is to deform $F$ to the coordinate subbundle $\epsilon^{d+k-1}$ inside $\epsilon^{d+k}$. 

 The sphere $S^{d+k-1}$ is $(d+k-2)$-connected and $d+k-2>k=\dim(M)$ by our hypothesis on $d$. Thus we can apply cellular approximation to conclude that $\tau$ is homotopic to the trivial map with image $e_{d+k}$. Let $\tau_t$ denote the corresponding homotopy. This gives a rank-$(d+k-1)$ bundle $\mathcal{F}$ over $M\times I$ whose fiber at $(p,t)$ is $\mathcal{F}_{(p,t)}=\langle \tau_t(p)\rangle^\perp$. By homotopy invariance of vector bundles it follows that $\mathcal{F}\cong pr_M^\ast F$, where $pr_M:M\times I\to M$ is projection. More precisely, we have a family of isomorphisms $\phi_{(p,t)}:F_p\to \mathcal{F}_{(p,t)} $ such that $\phi_{(p,0)}=Id$.

 Finally, we can homotope our original classifying map $f:M\to \Grtil_{d,n}$ via $f_t(p)=\phi_{(p,t)}(f(p))$, where orientation on $f_t(p)$ is chosen so that $\phi_{(p,t)}$ is orientation preserving. This gives the desired shortening. 
\end{proof}

Ultimately, we would like to construct an invariant \emph{of $E$} that tells us precisely when we can find a Parseval frame field on $E$ of size $d+k-1$. Proposition \ref{prop:euler_obstruction_for_parseval_frames} is not quite good enough because it only provides the obstruction to shortening \emph{a particular} Parseval frame field of size $d+k$.

Even though the isomorphism class of $E^\perp$ is not an invariant of $E$, it is still possible to extract an honest invariant of $E$.  

\begin{lemma}\label{lem:sw_class_is_invariant}
Let $E\to M$ be an oriented rank-$d$ vector bundle and let $E^\perp$ be the complement of $E$ coming from a particular choice of Parseval frame. The Stiefel–Whitney class $w(E^\perp)\in H^\ast(M,\ZZ_2)$ is independent of the choice of Parseval frame. More precisely, $w(E^\perp)=w(E)^{-1}$ as an element of the cohomology ring. In particular, $w(E^\perp)$ is an invariant of the oriented vector bundle $E$.
\end{lemma}

\begin{proof}

If $E^\perp$ comes from a Parseval frame of size $n$, then by construction $E\oplus E^\perp \cong \epsilon^n$. In other words, $E^\perp$ is a representative of the stable inverse of the vector bundle $E$.  Stiefel–Whitney classes behave multiplicatively under direct sums and so $w(E)w(E^\perp)=1$, and the result follows.
\end{proof}

Henceforth, we adopt the notation that $\wbar(E)=w(E)^{-1}$, and $\wbar_k(E)$ is the $k$th graded piece of $\wbar(E)$. From the discussion above, we see that the problem of finding a small Parseval frame can therefore be viewed as the problem of finding a low-rank representative of the stable inverse of $E$. Questions of this type have been studied classically. In particular, Malyi \cite{Malyi1974} studied the possible Euler classes of fixed-rank bundles within a prescribed stable equivalence class.

On the other hand, we have seen in our previous discussion that $w(E^\perp)$ and hence $w_k(E^\perp)$ is an honest invariant of $E$. Furthermore, since $w_k(E^\perp)$ is the mod 2 reduction of $e(E^\perp)$, it must vanish if $E\to M$ admits a Parseval frame of size $d+k-1$. We can now prove the following result that says $w_k(E^\perp)$ is the only obstruction to the existence of a Parseval frame of size $d+k-1$. 

\begin{theorem}\label{thm:d+k-1_obstruction}
    Let $E\to M$ be an orientable rank-$d$ vector bundle, with $d > 2$, over a closed, connected, orientable $k$-manifold. Then $E$ admits a Parseval frame field of size $n = d+k-1$ if and only if $\wbar_k(E)=0$.
\end{theorem} 

The proof will use the following construction for locally modifying a map, which is a particular instance of the construction described in~\cite[Theorem III.6.20]{whitehead2012elements}. Let $f:M\to N$ be a map into some manifold $N$, and let $[\nu] \in \pi_k(N)$ be a homotopy class represented by a map $\nu:S^k \to N$. For a closed disk $\Delta^k \subset M$, there exists a map $g:\Delta^k \to N$ such that $g|_{\partial \Delta^k} = f|_{\partial \Delta^k}$ and the map $S^k \to N$ obtained by gluing $f|_{\Delta^k}$ and $g$ along their common boundary $S^{k-1} \approx \partial \Delta^k$ is homotopic to $\nu$. We define the \emph{surgery of $f$} to be the map $f_\nu:M \to N$ defined by 
\[
f_\nu(m) \coloneqq \left\{\begin{array}{rl}
g(m) & \mbox{if } m \in \Delta^k \\
f(m) & \mbox{if } m \in M \setminus \Delta^k.
\end{array}\right.
\]
It will be convenient to use a different expression for this map, up to homotopy. Given $\Delta^k \subset M$, define the \emph{pinch map}
\[
p:M \to M \vee S^k
\]
to be the map obtained by collapsing the boundary $S^{k-1}  = \partial \Delta^k$ to a chosen basepoint~(cf.~\cite[Section 9.5]{May1999}). Then 
\[
f_\nu \simeq (f\vee \nu)\circ p.
\]
The proof of Theorem \ref{thm:d+k-1_obstruction} will use the following technical lemma about this construction.

\begin{lemma}
\label{lem:change_of_euler}
Let $M$ be a closed, connected, oriented $k$-manifold, let $N$ be a connected manifold, and let $\xi \to N$ be an oriented real vector bundle of rank-$k$. Let $f:M\to N$ be a map and let $[\nu]\in \pi_k(N)$. Suppose that $f_\nu:M\to N$ is obtained from $f$ by surgery along an embedded $k$-disk $\Delta^k$ in $M$. Then
\[
\left\langle e(f_\nu^*\xi),[M]\right\rangle
=
\left\langle e(f^*\xi),[M]\right\rangle
\pm
\left\langle e(\nu^*\xi),[S^k]\right\rangle.
\]
The sign depends on the orientation convention used in defining the surgery.
\end{lemma}

\begin{proof}
Recall that
\[
f_\nu \simeq (f\vee \nu)\circ p,
\]
where $p$ is the pinch map. Under the canonical identification
\[
H_k(M\vee S^k;\mathbb Z)
\cong
H_k(M;\mathbb Z)\oplus H_k(S^k;\mathbb Z),
\]
we have
\[
p_*[M]=[M]\pm[S^k],
\]
where the degree $\pm1$ of the second summand depends on the orientation convention. Using basic properties of Euler classes and wedge products, we have 
\begin{align*}
\left\langle e(f_\nu^*\xi),[M]\right\rangle
&=
\left\langle f_\nu^*e(\xi),[M]\right\rangle \\
&=
\left\langle e(\xi),(f_\nu)_*[M]\right\rangle \\
&=
\left\langle
e(\xi),
(f\vee \nu)_*\bigl([M]\pm[S^k]\bigr)
\right\rangle \\
&=
\left\langle e(\xi),f_*[M]\right\rangle
\pm
\left\langle e(\xi),\nu_*[S^k]\right\rangle \\
&=
\left\langle e(f^*\xi),[M]\right\rangle
\pm
\left\langle e(\nu^*\xi),[S^k]\right\rangle.
\end{align*}
\end{proof}

\begin{proof}[Proof of Theorem \ref{thm:d+k-1_obstruction}]
    From the previous discussion it is clear that $\wbar_k(E)=0$ is a necessary condition for the existence of such a frame field. Our goal is then to establish sufficiency.
    
    We first consider the case that $k$ is odd, in which case $\wbar_k(E)=0$ holds automatically. This case is essentially trivial: Lemma \ref{lem:euler_is_torsion} tells us that all oriented rank-$k$ bundles over $M$ have trivial Euler class, so for any particular Parseval frame field with classifying map $f:M\to \Gr_{d,d+k}$, we have $e(E^\perp)=0$, and Proposition \ref{prop:euler_obstruction_for_parseval_frames} shows that there is no obstruction to shortening any Parseval frame field of size $d+k$ to have size $d+k-1$. 
    
    The rest of the proof assumes that we are in the interesting case, where $k$ is even. The strategy is as follows. Choose an arbitrary Parseval frame of size $d+k$ and let $f:M\to \Grtil_{d,d+k}$ be the associated classifying map for $E$. To make explicit the dependence on $f$, we will denote the corresponding complement as $E_f^\perp$. By hypothesis, $\wbar_k(E)=0$. Since $\wbar_k(E)$ is the mod 2 reduction of $e(E^\perp_f)$, it follows that $\langle e(E_f^\perp), [M]\rangle$ is an even integer. Our goal is to define a new map $f_\sigma:M \to \Grtil_{d,d+k}$ satisfying
    \begin{equation}\label{eqn:goal_conditions}
        f_\sigma^\ast \tau_{d,d+k} \cong E \qquad \mbox{and} \qquad 
        \langle e(E_{f_\sigma}^\perp), [M] \rangle = \langle e(E_f^\perp), [M]\rangle \pm 2.
    \end{equation}
    Repeating the construction, this yields a classifying map for $E$ whose complementary bundle has trivial Euler class, so that Proposition \ref{prop:euler_obstruction_for_parseval_frames} can be applied to complete the proof. Our goal is then to define a classifying map $f_\sigma$ with the desired properties \eqref{eqn:goal_conditions}.

    Toward establishing the goal described above, we define a map $\sigma:S^k \to \Grtil_{d,d+k}$ as 
    \[
    \sigma(x) = \RR^{d-1} \oplus \langle x\rangle \subset \RR^{d-1} \oplus \RR^{k+1} \approx \RR^{d+k},
    \]
    where $(e_1,\ldots,e_{d-1},x)$ is declared to be a positively oriented basis. Before proceeding, let us record some facts about $\sigma$. First, $\sigma^\ast(\tau_{d,d+k})$ admits $d$ pointwise linearly independent sections: $x \mapsto e_i$ for $i=1,\ldots,d-1$ and $x \mapsto x$. 
    It follows that 
    \begin{equation}\label{eqn:sigma_properties_1}
        \sigma^\ast (\tau_{d,d+k}) \cong \epsilon^d, 
    \end{equation}
    Second, $\sigma(x)^\perp=\langle x\rangle ^\perp \cap \RR^{k+1}$, which is precisely the tangent plane to $S^k$ at $x$. It follows that $\sigma^\ast(\tau_{d,d+k}^\perp)\cong TS^k$, whence
    \begin{equation}\label{eqn:sigma_properties_2}
        \langle e(\sigma^\ast \tau^\perp_{d,d+k}), [S^k] \rangle = 2.
    \end{equation}
    Now we define $f_\sigma$ by surgery along a small disk $\Delta^k$, with respect to $[\sigma]$. 
    
    It remains to check that $f_\sigma$ satisfies \eqref{eqn:goal_conditions}. First, \eqref{eqn:sigma_properties_1} implies that 
    \[
    S^k \xrightarrow{\sigma} \Grtil_{d,d+k} \subset \Grtil_d
    \]
    is nullhomotopic as a map into $\Grtil_d$, so that $f_\sigma \simeq (f \vee \sigma) \circ p$ is homotopic to $f$ when both are considered as maps into $\Grtil_d$, hence $f$ and $f_\sigma$ classify the same bundle up to isomorphism, establishing the first condition in \eqref{eqn:goal_conditions}. Finally, the second condition of \eqref{eqn:goal_conditions} follows by Lemma \ref{lem:change_of_euler} and \eqref{eqn:sigma_properties_2}.
\end{proof}

\subsection{Equal-norm Parseval frame fields}

There is a similar obstruction theoretic picture for equal norm Parseval frames. Let $E\to M$ be an oriented rank-$d$ vector bundle over a $d$-manifold. Suppose $E$ admits an equal norm Parseval frame of size $2d$. Once again we get an isometric embedding $E\hookrightarrow M\times \RR^{2d}$ and can take the complement $E^\perp$, which is now another rank-$d$ vector bundle. Since $E$ admits an equal norm Parseval frame field and hence a nowhere zero section it follows that $e(E)=0$. It turns out that $E^\perp$ also has trivial Euler class. 

\begin{proposition}\label{prop:perp_has_trivial_euler_class}
    Let $n=2d$ or $2d+1$ and let $E$ be as above. Suppose that $E$ admits an equal norm Parseval frame field of size $n$. If $E^\perp$ is the corresponding complement, then $e(E^\perp)=0$. Furthermore, if $n=2d+1$ then $w_{d}(E^\perp)=0$.
\end{proposition}

\begin{remark}
Notice that in this case $E^\perp$ is a rank-$(n-d)$ bundle and the Euler class lives in $H^{n-d}(M)$. If $n\geq 2d+1$ then $n-d>d$ and this group vanishes for dimensional reasons. Thus the interesting cases occur when $n\leq 2d$.
\end{remark}

\begin{proof}
    Orthogonally projecting the standard basis vector fields from $M\times \RR^{n}$ onto $E^\perp$ gives a Parseval frame field of size $n$. We claim that this frame field is also equal norm, in which case it follows that $e(E^\perp)=0$ since it admits a nowhere zero section. To prove the claim, we let $m\in M$ and $P_m$ be the orthogonal projection matrix from $\RR^{n}$ to the fiber $E_m$. Let $P^\perp_m$ be the corresponding projection matrix onto $E^\perp_m$. We then have the identity $P_m+P^\perp_m=I$. The length squared of the projection of the $i$th standard basis vector onto $E^\perp_m$ is the matrix entry $(P^\perp_m)_{ii}=1-(P_m)_{ii}=1-d/n=\frac{n-d}{n}$. Thus the corresponding frame field is in fact equal norm.  

In the $n=2d+1$ case the above construction provides an equal norm Parseval frame field of size $n$ on $E^\perp$. The equal norm condition gives  $\abs{\sigma_{i}(m)}^2=\frac{d+1}{2d+1}>\frac{1}{2}$ for each $m\in M$.  Let $\sigma_i$ and $\sigma_j$ be two of the sections coming from the equal norm Parseval frame. Suppose for contradiction that there was some $m\in M$, where $\sigma_i(m)$ and $\sigma_j(m)$ are linearly dependent. Since $\sigma_i(m)$ and $\sigma_j(m)$ have the same length we have $\sigma_i(m)=\pm \sigma_j(m)$. Next, let $u=\displaystyle \frac{\sigma_i(m)}{\abs{\sigma_i(m)}}$. Since $u$ is a unit vector and the frame field is Parseval we compute
\[1=\sum_{k=1}^{2d+1}\abs{\langle u,\sigma_k(m)\rangle}^2
    \geq \abs{\langle u,\sigma_i(m)\rangle}^2+\abs{\langle u,\sigma_j(m)\rangle}^2=\frac{d+1}{2d+1}+\frac{d+1}{2d+1}>1\]
This gives a contradiction. Hence $\sigma_i$ and $\sigma_j$ are pointwise linearly independent. These linearly independent sections allow us to split off a 2-dimensional trivial bundle and write $E^\perp\cong F\oplus \epsilon^2$. Since $E^\perp$ and $F$ are stably equivalent we have $w_d(E^\perp)=w_d(F)=0$ since $F$ has rank $d-1$.
\end{proof}

\subsection{\texorpdfstring{$n=2d$}{n=2d} case}

 We now prove Theorem~\ref{thm:critical_dimension_theorem}. Recall that in this case, $\pi_d(\Grtil_{d,2d},\E_{d,2d})\cong \ZZ^{2n}$ (see Prop \ref{prop:homotopy_of_stiefel_pair}), and thus has additional generators compared to the $n\geq 2d+2$ case. We have seen that these additional generators come from meridional disks that encircle the $N_i$ from Lemma \ref{lem:enumerate_top_dim_pieces}. Intuitively, each of these strata consists of collections of vectors where the $i$th vector is essential in the sense that removing it drops the rank of the analysis operator from $d$ to $d-1$. Just as the $Z_i$ can be viewed as the zero locus of a section of the tautological bundle, we will show that the $N_i$ are corresponding zero sets of sections into the complement to the tautological bundle.  

For each $i$ we can define a section $\tau_i:\Grtil_{d,2d}\to \tau_{d,2d}^\perp$ by defining $\tau_i(P)$ to be the orthogonal projection of the $i$th standard basis vector to $P^\perp$. A point in $N_i$ corresponds to a $d$-plane $F$ in $\RR^{2d}$ such that the orthogonal projection of the standard basis elements indexed by $J_i$ fails to span $F$. Let $P_F$ be the corresponding orthogonal projection matrix. Since the vectors $\{P_F(e_k)\}_{k\in J_i}$ fail to span, there is a non-zero vector $v\in F$ that is orthogonal to all of them. Computing we find
\[
0=\langle v,P_F(e_k)\rangle=\langle P_F(v),e_k\rangle=\langle v,e_k\rangle. 
\]
In other words $v$ is orthogonal to all the standard basis vectors except for the $i$th one and is thus a multiple of $e_i$. This means that we can reinterpret $N_i$ as consisting of $d$-planes that contain $e_i$. If $e_i$ is contained in $F$ then it orthogonally projects to 0 in $F^\perp$, and so we see that $N_i$ is the zero locus of $\tau_i$. 

This gives a nice way to interpret $e(E^\perp)$. Given a classifying map $f:M\to \Grtil_{d,2d}$ coming from a particular Parseval frame, we form the complement bundle $E^\perp$ via pullback of $\tau_{d,2d}^\perp$. The Euler class of $E^\perp$ measures our inability to deform the original frame so that the $i$th section is nowhere essential, in the sense that if it is omitted then we still have a frame field. When $n\geq 2d+1$ this Euler class vanishes for dimensional reasons and we see that there is never an obstruction to deforming any particular section to be inessential. Given our discussion above we can prove a generalization of Lemma \ref{lem:relation_to_Euler} that we will state after introducing some notation. 

 For $1\leq i\leq 4d$ we have corresponding projection maps $$p_i:H^d(\Grtil_{d,2d},\E_{d,2d};\pi_d(\Grtil_{d,2d},\E_{d,2d}))\to H^d(\Grtil_{d,2d},\E_{d,2d};\ZZ).$$ 
We use the same notation $p_i$ for the induced projections on $H^d(M;\ZZ^{4d})$. Orient the meridians around $Z_i$ and $N_i$ using their identifications with the fibers of $\tau_{d,2d}$ and $\tau_{d,2d}^{\perp}$, respectively. Let $E\to M$ be an orientable rank-$d$ vector bundle over a $d$-manifold $M$ and let $f:M\to \Grtil_{d,2d}$ be a classifying map (or equivalently a choice of Parseval frame field of size $2d$ on $E$). As in the $n\geq 2d+2$ case we have a single obstruction $o_{ENF}(f)=f^\ast(o_{ENF})\in H^d(M;\pi_d(\Grtil_{d,n},\E_{d,n}))$. This class is the unique obstruction to deforming $f$ so that it has image in $\E_{d,n}$, or equivalently, deforming the corresponding Parseval frame to be equal norm. One subtlety that arises in this setting is that $o_{ENF}(f)$ is not an invariant of $E$ but of the classifying map $f$. This is essentially just a reframing of the fact that the complementary bundle $E^\perp$ is not an invariant of the bundle $E$, but depends on a choice of Parseval frame. 

\begin{lemma}\label{lem:relation_to_dual_euler}
    Let $n=2d$, $d>2$, let $E\to M$ be an orientable rank-$d$ vector bundle, and $\{\sigma_1,\ldots,\sigma_{2d}\}$ a Parseval frame field on $E$. Let $E^\perp$ be the complementary bundle in $M\times \RR^{2d}$ corresponding to this frame. 
    \begin{itemize}
        \item For $1\leq i\leq n$, $p_i(o_{ENF}(f))=e(E)$.
        \item For $n+1\leq i\leq 2n$, $p_i(o_{ENF}(f))=e(E^\perp)$
        \item $\cap_i \ker p_i=0$
    \end{itemize}
\end{lemma}

\begin{proof}
    The proof of the first and third bullet points is essentially the same as the proof of Lemma \ref{lem:relation_to_Euler}. For $2d+1\leq i\leq 4d$ our previous discussion shows that the $i$th generator of $\pi_d(\Grtil_{d,n},\E_{d,n})$ is a meridian around the zero locus of a transverse section of the universal bundle $\tau_{d,2d}^\perp$, hence exactly as in the case of the $Z_i$, the corresponding projection is just computing the Euler class $e(E^\perp)$.
\end{proof}

This lemma has the following corollary.
\begin{corollary}\label{cor:deform_frame_to_equal_norm_2d}
    Let $d>2$ and $n=2d$ and suppose that $E\to M$ is an orientable rank-$d$ vector bundle over $M$. Suppose that $\{\sigma_1,\ldots,\sigma_{2d}\}$ is a Parseval frame field for $E$. We can deform this frame through Parseval frames to be equal norm if and only if $e(E)=0$ and $e(E^\perp)=0$, where $E^\perp$ is the complement of $E$ determined by the original Parseval frame field.
\end{corollary}

\begin{proof}
By the preceding discussion, the Euler class restrictions are necessary for the frame to be deformed to have equal norm.

Let $f:M\to \Grtil_{d,2d}$ be the classifying map coming from the frame field $\{\sigma_1,\ldots,\sigma_{2d}\}$. The only obstruction to deforming this frame to be equal norm is $o_{ENF}(f)=f^\ast(o_{ENF})\in H^d(M,\pi_d(\Grtil_{d,2d},\E_{d,2d}))$. By hypothesis, $e(E)=e(E^\perp)=0$, and so by Lemma \ref{lem:relation_to_dual_euler} it follows that $p_i(o_{ENF}(f))=0$ for $1
\leq i\leq 4d$, and hence by applying Lemma \ref{lem:relation_to_dual_euler} again it follows that $o_{ENF}(f)=0$ and the frame can be deformed to be equal norm.
\end{proof}

As in the previous section we can show that the top Stiefel–Whitney class is the only additional obstruction to finding an equal norm Parseval frame. The argument again uses a surgery construction to drive the Euler class of $E^\perp$ to zero. We're now in a position to prove Theorem~\ref{thm:critical_dimension_theorem}, which we restate for convenience.

\begin{theorem}\label{thm:2d_obstruction}
    Let $d>2$, $n=2d$, and let $E\to M$ be an orientable rank-$d$ vector bundle over a closed orientable $d$-manifold. $E$ admits an equal-norm Parseval frame field of size $2d$ if and only if $e(E)=0$ and $\wbar_d(E)=0$.
\end{theorem}

\begin{proof}
    From the previous discussion it is clear that $e(E)=\wbar_d(E)=0$ is a necessary condition for the existence of such a frame field.

    Again, if $d$ is odd Lemma \ref{lem:euler_is_torsion} tells us that all oriented rank-$d$ bundles over $M$ have trivial Euler class. This means that for any particular Parseval frame field $e(E)=e(E^\perp)=0$, and Corollary \ref{cor:deform_frame_to_equal_norm_2d} shows that there is no obstruction to deforming any Parseval frame to have equal norm. 
    
    Henceforth, we assume that $d$ is even.  Let $f:M\to \Grtil_{d,2d}$ be a classifying map for $E$. To make explicit the dependence on $f$ we will denote the corresponding complement as $E_f^\perp$. By hypothesis, $\wbar_d(E)=0$. Since $\wbar_d(E)$ is the mod 2 reduction of $e(E^\perp_f)$, it follows that $e(E_f^\perp)$ is even. 

This time we construct a similar map $\sigma:S^d\to \Grtil_{d,2d}$. Split $\RR^{2d}=\RR^{d-1}\times \RR^{d+1}$. Using this splitting and the standard embedding of $S^d$ into $\RR^{d+1}$ gives an embedding of $S^d$ into $\RR^{2d}$ and we henceforth implicitly regard $S^d$ as a subset of $\RR^{2d}$. For each $x\in S^d$ we can form a $d$-plane $\RR^{d-1}\oplus \langle x\rangle$. Declaring $\{e_1,\ldots,e_{d-1},x\}$ to be a positively oriented basis provides a map $\sigma:S^d\to \Grtil_{d,2d}$. Arguing as before we see that $\sigma^\ast(\tau_{d,2d})$ is trivial and $\sigma^\ast(\tau_{d,2d}^\perp)\cong TS^d$.

Applying surgery corresponding to $\sigma$ to the classifying map $f$ and arguing as in the proof of Theorem \ref{thm:d+k-1_obstruction} shows that we can eventually drive the Euler class of $E^\perp$ to 0 and apply Corollary \ref{cor:deform_frame_to_equal_norm_2d} to produce the desired equal norm Parseval frame field. 
\end{proof}

\section{Examples and non-existence results}\label{sec:examples}

Ultimately, we would like to use Propositions \ref{prop:euler_obstruction_for_parseval_frames} and \ref{prop:perp_has_trivial_euler_class} to find examples of bundles that do not admit Parseval frame fields of particular sizes. On its face, this looks challenging since Remark \ref{rem:perp_depends_on_frame} suggests that different Parseval frames may give rise to non-equivalent complements. However, there is an honest invariant that can be extracted. 

Using these ideas we can construct new counterexamples. Let $M=\CP^k$ be complex projective $k$-space and let $\xi$ be the tautological line bundle over $\CP^k$. The total Stiefel–Whitney class of $\xi$ is $1+x$, where $0\neq x\in H^2(\CP^k;\ZZ_2)\cong \ZZ_2$ is the mod 2 reduction of the first Chern class. Next, let $E_d=\xi\oplus \RR^{d-2}$, which we regard as a real oriented rank-$d$ vector bundle over a real manifold of dimension $2k$. If $d>2$, it follows that $E_d$ admits a nowhere-zero section, so $e(E_d)=0$. Furthermore, since the Stiefel–Whitney class is a stable invariant we have $w(E_d)=1+x$.

\begin{theorem}
    Let $d> 2$ and let $E_d\to M$ be as above. 
    \begin{itemize}
        \item $E_d$ does not admit a Parseval frame field of size $2k+d-1$.
        \item $E_{2k}$ does not admit an equal norm Parseval frame field of size $4k$ or $4k+1$.
    \end{itemize}

\end{theorem}

\begin{proof}
    Suppose for contradiction that $E_d$ admits a Parseval frame of size $2k+d-1$. By Theorem \ref{thm:d+k-1_obstruction} we have $\wbar_{2k}(E_d)=0$. 

    However, $w(E_d)\wbar(E_d)=1$. Since $w(E_d)=1+x$, it follows that $\wbar(E_d)=1+x+x^2+\ldots+x^{k}$. In particular, $\wbar_{2k}(E_d)\neq 0$, which is a contradiction.

   The same computation, together with Proposition \ref{prop:perp_has_trivial_euler_class}, gives a similar contradiction in the equal norm setting. If an equal norm Parseval frame of size $4k$ (resp.\ $4k+1$) exists then we have $\wbar_{2k}(E_{2k})=0$. However this gives a similar contradiction as above since $\wbar_{2k}(E_{2k})=x^k\neq 0$.
\end{proof}

The examples above show that the dimensional restrictions of \cite[Theorem 4.1]{BallasNeedhamShonkwiler2025} and Theorem \ref{thm:main_thm} are sharp, in the sense that it is always possible to find an orientable vector bundle over an even-dimensional orientable base that doesn't admit a Parseval or equal norm Parseval frame field when the dimensional restrictions of the theorem are violated. Over an odd-dimensional orientable base the obstructions to finding equal norm Parseval frames coming from Theorem \ref{thm:2d_obstruction} vanish by the odd-rank Euler-class argument above, and so these arguments do not produce analogous counterexamples over odd-dimensional orientable bases. 

If one relaxes the orientability condition on the base then it is possible to produce orientable bundles over odd-dimensional bases that do not admit Parseval frames of size $d+k-1$.

The first such example occurs when \(k=d=3\). Let \(M=\mathbb{RP}^1\times\mathbb{RP}^2\), and let \(\xi_1\) and \(\xi_2\) denote the pullbacks of the tautological real line bundles over \(\mathbb{RP}^1\) and \(\mathbb{RP}^2\), respectively, under the projection maps. Consider the bundle \(E=\xi_1\oplus\xi_2\oplus(\xi_1\otimes\xi_2)\).

The cohomology ring with $\ZZ_2$ coefficients of \(M\) is \(H^*(M;\mathbb Z_2)\cong\mathbb Z_2[x,y]/(x^2,y^3)\),
where \(x\) and \(y\) are degree-$1$ generators coming from $\mathbb{RP}^1$ and $\mathbb{RP}^2$, respectively. Since \(w(\xi_1)=1+x\), \(w(\xi_2)=1+y\), and \(w(\xi_1\otimes\xi_2)=1+x+y\), the Whitney product formula gives $$w(E)=(1+x)(1+y)(1+x+y)=1+xy+y^2+xy^2.$$ In particular, \(w_1(E)=0\), so \(E\) is orientable. On the other hand, $\mathbb{RP}^2$ is non-orientable, so $M$ is a non-orientable manifold. Since \(w(E)^2=1\), we have \(\overline{w}(E)=w(E)\), and hence $\overline{w}_3(E)=xy^2\neq0$.
If such a frame existed, its rank-$2$ complement would have $w_3(E^\perp)=\wbar_3(E)\neq0$, a contradiction. Thus \(E\) does not admit a Parseval frame field of size \(d+k-1=5\). Taking Whitney sums with trivial bundles extends this example to all \(d\ge3\) when \(k=3\).

We can also construct examples for every odd \(k\ge5\). Write \(k=2t+3\), where \(t\ge1\), and let $M_k=\mathbb{RP}^1\times\mathbb{RP}^2\times\mathbb{CP}^{t}$. Let \(\xi_3\) be the pullback of the underlying real rank-$2$ bundle of the tautological complex line bundle over \(\mathbb{CP}^t\). Consider $F=\xi_1\oplus\xi_2\oplus(\xi_1\otimes\xi_2)\oplus\xi_3$.
The cohomology ring with $\ZZ_2$ coefficients is $H^*(M_k;\mathbb Z_2)\cong
\mathbb Z_2[x,y,z]/(x^2,y^3,z^{t+1})$, where \(\deg x=\deg y=1\) and \(\deg z=2\). Since \(w(\xi_3)=1+z\), we obtain $$w(F)=(1+xy+y^2+xy^2)(1+z).$$
Thus \(w_1(F)=0\), so the vector bundle is orientable although $M_k$ is non-orientable. Moreover,
$$\overline{w}(F)=(1+xy+y^2+xy^2)(1+z+z^2+\cdots+z^t),$$
and consequently $\overline{w}_k(F)=xy^2z^t\neq0,$ since $xy^2z^t$ is a generator of \(H^k(M_k;\mathbb Z_2)\). Finally, for any \(d\ge5\), taking \(E_d=F\oplus\epsilon^{d-5}\) gives an orientable rank-\(d\) vector bundle over the nonorientable \(k\)-manifold \(M_k\) satisfying $\overline{w}_k(E_d)\neq0$.

\bibliographystyle{plain}
\bibliography{bibliography}

@article{monaco2023topology,
  title={Topology vs localization in synthetic dimensions},
  author={Monaco, Domenico and Roussign{\'e}, Thaddeus},
  journal={Journal of Mathematical Physics},
  volume={64},
  number={1},
  year={2023},
  publisher={AIP Publishing}
}

@article{Atiyah1982,
  author  = {Atiyah, Michael F.},
  title   = {Convexity and Commuting Hamiltonians},
  journal = {Bulletin of the London Mathematical Society},
  volume  = {14},
  number  = {1},
  pages   = {1--15},
  year    = {1982},
  doi     = {10.1112/blms/14.1.1}
}

@article{Woodward2010,
  author  = {Woodward, Chris},
  title   = {Moment Maps and Geometric Invariant Theory:
             Corrected Version (October 2011)},
  journal = {Les cours du CIRM},
  volume  = {1},
  number  = {1},
  pages   = {121--166},
  year    = {2010},
  doi     = {10.5802/ccirm.29},
  url     = {https://www.numdam.org/articles/10.5802/ccirm.29/}
}

@article{cornean2019parseval,
  title={Parseval frames of exponentially localized magnetic Wannier functions},
  author={Cornean, Horia D. and Monaco, Domenico and Moscolari, Massimo},
  journal={Communications in Mathematical Physics},
  volume={371},
  number={3},
  pages={1179--1230},
  year={2019},
  publisher={Springer}
}

@article{packer2004projective,
  title={Projective multi-resolution analyses for {$L^2(\mathbb{R}^2)$}},
  author={Packer, Judith A. and Rieffel, Marc A.},
  journal={Journal of Fourier Analysis and Applications},
  volume={10},
  number={5},
  pages={439--464},
  year={2004},
  publisher={Springer}
}

@article{favazza2026phase,
  title={Phase retrieval and approximation of sections of vector bundles from sampled frame coefficients},
  author={Favazza, Kieran and Frank, Olivia and Freeman, Daniel and Ghoreishi, Dorsa and King, Bernadette and Randrianarivony, Nirina Lovasoa},
  journal={La Matematica},
  year={2026},
  publisher={Springer Nature},
  note={To appear}
}

@article{frank2002frames,
  title={Frames in {Hilbert} {$C^*$}-modules and {$C^*$}-algebras},
  author={Frank, Michael and Larson, David R.},
  journal={Journal of Operator Theory},
  pages={273--314},
  year={2002},
  publisher={JSTOR}
}

@book{whitehead2012elements,
  title={Elements of Homotopy Theory},
  author={Whitehead, George W.},
  year={2012},
  publisher={Springer Science \& Business Media}
}

@article{BallasNeedhamShonkwiler2025,
  author  = {Samuel A. Ballas and Tom Needham and Clayton Shonkwiler},
  title   = {On the Existence of {Parseval} Frames for Vector Bundles},
  journal = {Transactions of the American Mathematical Society, Series B},
  volume  = {12},
  year    = {2025},
  pages   = {395--416},
  doi     = {10.1090/btran/223},
  mrnumber = {4883097},
  zbl     = {1564.42022}
}

@article{CaineNeedhamShonkwiler2025,
  author       = {Anthony Caine and Tom Needham and Clayton Shonkwiler},
  title        = {Optimization and the Topology of Spaces of {Parseval} Frames},
  journal      = {SIAM Journal on Matrix Analysis and Applications},
  volume       = {47},
  number       = {3},
  year         = {2026},
  pages        = {1375--1399}
}

@book{BottTu1982,
  author    = {Raoul Bott and Loring W. Tu},
  title     = {Differential Forms in Algebraic Topology},
  series    = {Graduate Texts in Mathematics},
  volume    = {82},
  publisher = {Springer-Verlag},
  address   = {New York},
  year      = {1982},
  isbn      = {978-0-387-90613-3}
}

@book{milnor1974characteristic,
  title={Characteristic classes},
  author={Milnor, John Willard and Stasheff, James D.},
  volume={76},
  series={Annals of Mathematics Studies},
  year={1974},
  publisher={Princeton University Press},
  address={Princeton, NJ, USA}
}

@book{DolgachevInvariantTheory,
  author    = {Dolgachev, Igor V.},
  title     = {Lectures on Invariant Theory},
  series    = {London Mathematical Society Lecture Note Series},
  volume    = {296},
  publisher = {Cambridge University Press},
  address   = {Cambridge},
  year      = {2003},
  doi       = {10.1017/CBO9780511615436},
  isbn      = {978-0-521-52548-0}
}

@article{AmendolaKohnReichenbachSeigal,
  author  = {Am{\'e}ndola, Carlos and Kohn, Kathl{\'e}n and
             Reichenbach, Philipp and Seigal, Anna},
  title   = {Toric invariant theory for maximum likelihood
             estimation in log-linear models},
  journal = {Algebraic Statistics},
  volume  = {12},
  number  = {2},
  pages   = {187--211},
  year    = {2021},
  doi     = {10.2140/astat.2021.12.187}
}

@book{May1999,
  author    = {May, J. Peter},
  title     = {A Concise Course in Algebraic Topology},
  series    = {Chicago Lectures in Mathematics},
  publisher = {University of Chicago Press},
  address   = {Chicago, IL},
  year      = {1999}
}

@book{GuilleminPollack1974,
  author    = {Guillemin, Victor and Pollack, Alan},
  title     = {Differential Topology},
  publisher = {Prentice-Hall},
  address   = {Englewood Cliffs, NJ},
  year      = {1974}
}

@article{Malyi1974,
  author  = {Malyi, Boris D.},
  title   = {Euler Classes of Stably Equivalent Vector Bundles},
  journal = {Mathematics of the USSR-Izvestiya},
  volume  = {8},
  number  = {1},
  year    = {1974},
  pages   = {113--131},
  doi     = {10.1070/IM1974v008n01ABEH002098}
}

@book{FultonYoungTableaux,
  author    = {Fulton, William},
  title     = {Young Tableaux: With Applications to Representation Theory and Geometry},
  series    = {London Mathematical Society Student Texts},
  volume    = {35},
  publisher = {Cambridge University Press},
  address   = {Cambridge},
  year      = {1997},
  isbn      = {978-0-521-56724-4}
}

@article{LermanMomentMap,
  author  = {Lerman, Eugene},
  title   = {Gradient flow of the norm squared of a moment map},
  journal = {Enseign. Math. (2)},
  volume  = {51},
  number  = {1-2},
  year    = {2005},
  pages   = {117--127}
}

@book{GeorgoulasRobbinSalamon,
  author    = {Georgoulas, Valentina and Robbin, Joel W. and Salamon, Dietmar A.},
  title     = {The Moment-Weight Inequality and the Hilbert--Mumford Criterion:
               {GIT} from the Differential Geometric Viewpoint},
  series    = {Lecture Notes in Mathematics},
  volume    = {2297},
  publisher = {Springer},
  address   = {Cham},
  year      = {2021},
  doi       = {10.1007/978-3-030-89300-2},
  isbn      = {978-3-030-89299-9}
}

@book{waldron2018introduction,
  title={An Introduction to Finite Tight Frames},
  author={Waldron, Shayne F. D.},
  series={Applied and Numerical Harmonic Analysis},
  year={2018},
  publisher={Birkhäuser},
  address={New York, NY, USA}
}

@article{casazza2003equal,
  title={Equal-norm tight frames with erasures},
  author={Casazza, Peter G. and Kova{\v{c}}evi{\'c}, Jelena},
  journal={Advances in Computational Mathematics},
  volume={18},
  pages={387--430},
  year={2003},
  publisher={Springer}
}

@article{bodmann2005frames,
  title={Frames, graphs and erasures},
  author={Bodmann, Bernhard G. and Paulsen, Vern I.},
  journal={Linear Algebra and its Applications},
  volume={404},
  pages={118--146},
  year={2005},
  publisher={Elsevier}
}

@book{Husemoller1994,
  author    = {Dale Husemoller},
  title     = {Fibre Bundles},
  edition   = {3},
  series    = {Graduate Texts in Mathematics},
  volume    = {20},
  publisher = {Springer-Verlag},
  address   = {New York},
  year      = {1994},
  isbn      = {978-0-387-94087-8}
}

@article{kuchment2016overview,
  title={An overview of periodic elliptic operators},
  author={Kuchment, Peter},
  journal={Bulletin of the American Mathematical Society},
  volume={53},
  number={3},
  pages={343--414},
  year={2016}
}

@article{auckly2018parseval,
  title={On {Parseval} frames of exponentially decaying composite {Wannier} functions},
  author={Auckly, David and Kuchment, Peter},
  journal={Mathematical Problems in Quantum Physics},
  volume={717},
  pages={227--240},
  year={2018},
  publisher={American Mathematical Society Providence, RI}
}

@article{kuchment2008tight,
  title={Tight frames of exponentially decaying {Wannier} functions},
  author={Kuchment, Peter},
  journal={Journal of Physics A: Mathematical and Theoretical},
  volume={42},
  number={2},
  pages={025203},
  year={2008},
  publisher={IOP Publishing}
}

@article{han2000frames,
  title={Frames, bases and group representations},
  author={Han, Deguang and Larson, David R.},
  journal={Memoirs of the American Mathematical Society},
  volume={147},
  year={2000},
  number={697},
}

@book{CasazzaKutyniok2013,
  title     = {Finite Frames: Theory and Applications},
  editor    = {Casazza, Peter G. and Kutyniok, Gitta},
  series    = {Applied and Numerical Harmonic Analysis},
  publisher = {Birkh{\"a}user},
  address   = {Boston},
  year      = {2013}
}

@article{freeman2012moving2,
  AUTHOR = {Freeman, Daniel and Hotovy, Ryan and Martin, Eileen},
     TITLE = {Moving finite unit norm tight frames for {$S^n$}},
   JOURNAL = {Illinois Journal of Mathematics},
    VOLUME = {58},
      YEAR = {2014},
    NUMBER = {2},
     PAGES = {311--322},
      ISSN = {0019-2082},
   MRCLASS = {42C15 (57R22)},
  MRNUMBER = {3367650},
MRREVIEWER = {Stephan Fackler},
       URL = {http://projecteuclid.org.proxy.lib.fsu.edu/euclid.ijm/1436275485}
}

@article{freeman2012moving,
  AUTHOR = {Freeman, Daniel and Poore, Daniel and Wei, Ann Rebecca and Wyse, Madeline},
     TITLE = {Moving {P}arseval frames for vector bundles},
   JOURNAL = {Houston Journal of Mathematics},
    VOLUME = {40},
      YEAR = {2014},
    NUMBER = {3},
     PAGES = {817--832},
      ISSN = {0362-1588},
   MRCLASS = {42C15 (57R22)},
  MRNUMBER = {3275625},
MRREVIEWER = {Stephan Fackler},
}

@book {Steenrod_Fibre_Bundles,
    AUTHOR = {Steenrod, Norman},
     TITLE = {The {T}opology of {F}ibre {B}undles},
    SERIES = {Princeton Mathematical Series, vol. 14},
 PUBLISHER = {Princeton University Press},
   ADDRESS = {Princeton, NJ, USA},
      YEAR = {1951},
     PAGES = {viii+224},
   MRCLASS = {56.0X},
  MRNUMBER = {0039258},
MRREVIEWER = {S. Chern},
}

@book {Hatcher_Algebraic_Topology,
    AUTHOR = {Hatcher, Allen},
     TITLE = {Algebraic Topology},
 PUBLISHER = {Cambridge University Press, Cambridge},
      YEAR = {2002},
     PAGES = {xii+544},
      ISBN = {0-521-79160-X; 0-521-79540-0},
   MRCLASS = {55-01 (55-00)},
  MRNUMBER = {1867354},
MRREVIEWER = {Donald W. Kahn},
}

@article{holmes2004optimal,
  title={Optimal frames for erasures},
  author={Holmes, Roderick B. and Paulsen, Vern I.},
  journal={Linear Algebra and its Applications},
  volume={377},
  pages={31--51},
  year={2004},
  publisher={Elsevier}
}

@article{vershynin2005frame,
  title={Frame expansions with erasures: an approach through the non-commutative operator theory},
  author={Vershynin, Roman},
  journal={Applied and Computational Harmonic Analysis},
  volume={18},
  number={2},
  pages={167--176},
  year={2005},
  publisher={Elsevier}
}

@article{goyal2001quantized,
  title={Quantized frame expansions with erasures},
  author={Goyal, Vivek K. and Kova{\v{c}}evi{\'c}, Jelena and Kelner, Jonathan A.},
  journal={Applied and Computational Harmonic Analysis},
  volume={10},
  number={3},
  pages={203--233},
  year={2001},
  publisher={Elsevier}
}

@article{benedetto2006second,
  title={Second-order Sigma--Delta ($\Sigma$$\Delta$) quantization of finite frame expansions},
  author={Benedetto, John J. and Powell, Alexander M. and Y{\i}lmaz, {\"O}zg{\"u}r},
  journal={Applied and Computational Harmonic Analysis},
  volume={20},
  number={1},
  pages={126--148},
  year={2006},
  publisher={Elsevier}
}

@article{bodmann2007frame,
  title={Frame paths and error bounds for sigma--delta quantization},
  author={Bodmann, Bernhard G. and Paulsen, Vern I.},
  journal={Applied and Computational Harmonic Analysis},
  volume={22},
  number={2},
  pages={176--197},
  year={2007},
  publisher={Elsevier}
}

@phdthesis{caineNormalizingParsevalFrames2024,
  title = {Normalizing {{Parseval Frames}} by {{Gradient Descent}}},
  author = {Caine, Anthony},
  year = 2024,
  eprint = {10217/239830},
  eprinttype = {hdl},
  address = {Fort Collins, CO, USA},
  doi = {10.25675/3.04572},
  collaborator = {Peterson, Chris and Shonkwiler, Clayton},
  hdl = {10217/239830},
  langid = {english},
  school = {Colorado State University}
}
\end{document}